\documentclass{cmslatex}

\usepackage[paperwidth=7in, paperheight=10in, margin=.875in]{geometry}
 \usepackage[backref,colorlinks,linkcolor=red,anchorcolor=green,citecolor=blue]{hyperref}
\usepackage{amsfonts,amssymb}
\usepackage{amsmath}
\usepackage{graphicx}
\usepackage{cite}
\usepackage{enumerate}
\renewcommand{\theequation}{\arabic{section}.\arabic{equation}}

\usepackage{lipsum}
\usepackage{amsfonts}
\usepackage{graphicx}
\usepackage{epstopdf}
\usepackage{algorithm}
\usepackage{algorithmic}
\usepackage{tikz}
\usetikzlibrary{arrows}
\usepackage{bm}
\usepackage{multirow}
\usepackage{listings}
\usepackage{mathtools}
\usepackage{latexsym,amsmath,amsfonts,amscd}
\usepackage{subfigure}
\usepackage{bbm}
\newtheorem{rmk}{Remark}
\newtheorem{defi}{Definition}

\ifpdf
  \DeclareGraphicsExtensions{.eps,.pdf,.png,.jpg}
\else
  \DeclareGraphicsExtensions{.eps}
\fi

\newcommand{\creflastconjunction}{, and~}

\title{Discrete Maximum principle of a high order finite difference scheme for a
generalized Allen-Cahn equation
\thanks{J. Shen is supported in part by NSF grant DMS-2012585 and AFOSR  grant FA9550-20-1-0309 while X. Zhang  is  supported in part by the NSF grant DMS-1913120. }}

\author{Jie Shen \thanks{Department of Mathematics,
Purdue University,
150 N. University Street,
West Lafayette, IN 47907-2067
  ({shen@math.purdue.edu}, {zhan1966@purdue.edu}).}
\and Xiangxiong Zhang \footnotemark[2]}

\usepackage{amsopn}

\begin{document}

         \pagestyle{myheadings} \markboth{Maximum principle of high order schemes}{J. Shen and X. Zhang}\maketitle

\begin{abstract}
We consider 
solving a generalized Allen-Cahn equation coupled with a passive convection for a  given incompressible velocity field. 
The  numerical scheme consists of the first order accurate stabilized implicit explicit time discretization and  a fourth order accurate finite difference scheme, which is obtained from the finite difference formulation of the $Q^2$ spectral element method. 
We prove that the discrete maximum principle holds under suitable mesh size and time step constraints. 
The same result also applies to construct a bound-preserving scheme for any passive convection with an incompressible velocity field. 

\end{abstract}



\section{Introduction}
In this paper
we consider bound-preserving schemes for a generalized Allen-Cahn equation 
\begin{equation}
\label{generalized-Allen-Cahn}
\phi_t+u\phi_x+v \phi_y=\mu \Delta \phi-\frac{F'(\phi)}{\varepsilon}, \quad (x,y)\in \Omega,\end{equation}
where $\Omega$ is an open bounded domain in $\mathbb{R}^2$, $\mu,\varepsilon>0$ are parameters,  $F(\phi)$ is an energy function, and $(u,v)$ is a given incompressible velocity field. The Allen-Cahn equation, i.e.\eqref{generalized-Allen-Cahn} with $(u,v)\equiv 0$, plays an important role in materials science \cite{All.C79,chen2002phase}. The generalized Allen-Cahn equation \eqref{generalized-Allen-Cahn}, often with an extra Lagrange multiplier to conserve the volume fraction \cite{yang2006numerical},  are frequently encountered in modeling of multi-phase 
incompressible flows, e.g., \cite{liu2003phase}. 

The generalized Allen-Cahn equation \eqref{generalized-Allen-Cahn} usually satisfies a maximum principle, so it is desired to have its numerical solution to preserve the  maximum principle or to remain in a prescribed bound.
In particular, this becomes  crucial when the energy function $F(\phi)$ is of the form
\begin{equation}
\label{log-energy}
F(\phi)=\frac{\theta}{2}[(1+\phi)\ln(1+\phi)+(1-\phi)\ln(1-\phi)]-\frac{\theta_c}{2}\phi
\end{equation}
where $\theta, \theta_c$ are two positive constants.

To construct bound-preserving schemes for equation \eqref{generalized-Allen-Cahn}, 
we can first consider bound-preserving schemes for a convection-diffusion equation, e.g., $F(\phi)\equiv 0$. 
In the literature, there are many fully explicit high order accurate bound-preserving schemes for a  scalar convection-diffusion equation  \cite{zhang2012maximum, chen2016third, srinivasan2018positivity, sun2018discontinuous, li2018high, qiu2021third, GLY19}. 
In these schemes,  the time discretizations are high order explicit time strong stability preserving (SSP) Runge-Kutta and multistep methods, which are convex combinations of forward Euler steps.  
Even though such an approach allows various high order  accurate spatial discretizations, all these fully explicit schemes require a small time step $\Delta t=\mathcal O(\frac{1}{\mu}\Delta x^2)$, which is inpractical unless $\mu$ is very small. 

To construct bound-preserving schemes without the parabolic type CFL constraint  $\Delta t=\mathcal O(\frac{1}{\mu}\Delta x^2)$ for \eqref{generalized-Allen-Cahn}, 
 the second order finite difference was used  in \cite{shen2016maximum}
 with an implicit explicit (IMEX) time discretization
\begin{equation}
\label{imex}
\frac{\phi^{n+1}-\phi^n}{\Delta t}+u^{n+1}\phi^{n+1}_x+v^{n+1} \phi^{n+1}_y=\mu \Delta \phi^{n+1}-\frac{F'(\phi^n)}{\varepsilon},
\end{equation}
and a stabilized scheme with a parameter $S\geq 0$:
\begin{equation}
\label{imex2}
\frac{\phi^{n+1}-\phi^n}{\Delta t}+S(\phi^{n+1}-\phi^n)+u^{n+1}\phi^{n+1}_x+v^{n+1} \phi^{n+1}_y=\mu \Delta \phi^{n+1}-\frac{F'(\phi^n)}{\varepsilon}. 
\end{equation}
For spatial discretization, it is well-known that the second order finite difference for \eqref{imex} and \eqref{imex2} forms an M-matrix \cite{shen2016maximum}, thus the matrix of the linear system in \eqref{imex} and  \eqref{imex2} is {\it monotone}, i.e., the inverse matrix is entrywise non-negative. 
Monotonicity is the key property which implies the discrete maximum principle. In general high order accurate schemes do not form 
 M-matrices thus   it is also quite challenging to extend the method in \cite{shen2016maximum} to higher spatial accuracy. Nonetheless, recent progress in \cite{li2019monotonicity} shows that  the finite element method with $Q^2$ polynomial (tensor product quadratic polynomial)  on structured meshes is a product of two M-matrices for a diffusion operator 
 thus is still monotone. 
 
The main purpose of  this paper is to extend the results in \cite{li2019monotonicity} to the spatial discretization for \eqref{imex} and  \eqref{imex2}. In particular, when $Q^2$ finite element method for a convection-diffusion operator in \eqref{imex} and  \eqref{imex2} is implemented with $3$-point Gauss-Lobatto quadrature as a finite difference scheme,   it can be rigorously proven that it is a fourth order accurate  spatial discretization in the discrete $l^2$-norm  \cite{li2020superconvergence, li2021accuracy}.
 In the literature, $Q^k$ finite element method implemented by  $m$-point  Gauss-Lobatto quadrature with $m\geq k+1$ is also called spectral element method \cite{maday1990optimal}.
 The fourth order finite difference scheme in this paper is also equivalent to $Q^2$ spectral element method with only $3$-point Gauss-Lobatto quadrature.
More precisely, we will prove that this fourth order finite difference spatial discretization for \eqref{imex} and  \eqref{imex2} satisfies the discrete maximum principle under certain mesh size and time step constraints. For the discrete maximum principle to hold for  a convection-diffusion equation, the time step constraint in this paper is a lower bound condition on $\frac{\Delta t}{\Delta x^2}$ thus still  practical.

 For extensions to higher order time accuracy, in general it is quite difficult  since there are no high order SSP implicit time discretizations without the constraint $\Delta t=\mathcal O(\frac{1}{\mu}\Delta x^2)$, see \cite{gottlieb2009high}.  For a second order spatial discretization, one possible approach to obtain a second order accurate time scheme is to consider 
the   exponential time differencing schemes \cite{du2019},  which heavily depends on the $\ell^\infty$ estimate of the matrix exponential $e^{-\Delta_h}$
with $-\Delta_h$ denoting the discrete Laplacian matrix.  For the second order finite difference, such an estimate can be established by the exact solution of the ordinary differential equations of the semi-discrete scheme for solving heat equation since the second order finite difference gives a diagonally dominant matrix $-\Delta_h$. Unfortunately,  if using the fourth order accurate spatial discretization in this paper, the discrete Laplacian matrix $-\Delta_h$ is no longer diagonally dominant.

The rest of the paper is organized as follows. In Section \ref{sec-scheme}, we review the finite difference scheme obtained from $Q^2$ spectral element method. 
Its monotonicity is proved in Section \ref{sec-mono}. 
In Section \ref{sec-ac}, we establish the discrete maximum principle for the generalized Allen-Cahn equation with both polynomial and logarithmic energy functions.
The main results in Section \ref{sec-mono} can also be used to construct a fourth order bound-preserving finite difference spatial discretization for  any passive convection. As a demonstration, we apply it to
the two-dimensional incompressible Navier-Stokes equation in stream function vorticity formulation in Section \ref{sec-ns}. 
We present in Section \ref{sec-test}  several numerical tests to validate our scheme. Some concluding remarks are given in Section \ref{sec-remark}.

\section{Finite difference implementation of $Q^2$ spectral element method}
\label{sec-scheme}
For simplicity, we  derive the scheme for an elliptic equation with incompressible velocity field $\mathbf u=(u,v)$  
and  a given function $f$ 
on a square domain $ \Omega=(0,1)\times(0,1)$ and
Dirichlet boundary conditions:
\begin{equation}
 \phi +\mathbf u\cdot\nabla \phi -\nabla\cdot(\mu\nabla \phi)= f  \textrm{ on } \Omega,  \quad \phi(x,y) =  g(x,y) \textrm{ on } \partial\Omega.
 \label{ellipticeqn}
 \end{equation}
We only consider a constant scalar $\mu$ and $Q^2$ elements on a uniform  mesh, even though the scheme can also be easily extended to general scenarios such as 
Neumann boundary conditions,  and
diffusion terms like a variable $\mu$ or $\nabla\cdot(A \nabla \phi)$  with a positive definite matrix function $A$, see  \cite{li2019fourth}.

Let $\Omega_h$ denote a uniform rectangular mesh as shown in  Figure \ref{mesh} (a). 
Let $Q^2(e)$ be the set of 
 tensor product of quadratic polynomials on a rectangular cell $e$:
 $$Q^2(e)=\left\{p(x,y)=\sum\limits_{i=0}^2\sum\limits_{j=0}^2 p_{ij} x^iy^j, (x,y)\in e\right\}.$$
 Let  $V^h$ and $V^h_0$ denote two continuous piecewise $Q^{2}$ finite element spaces on $\Omega_h$:
  $$V^h=\{p(x,y)\in C^0(\Omega_h): p|_e \in Q^{2}(e),\quad \forall e\in \Omega_h\},$$
  $$V^h_0=\{v_h\in V^h: v_h=0 \quad\mbox{on}\quad \partial \Omega \}.$$

   \begin{figure}[htbp]
 \subfigure[A rectangular mesh and quadrature points.]{\includegraphics[scale=0.8]{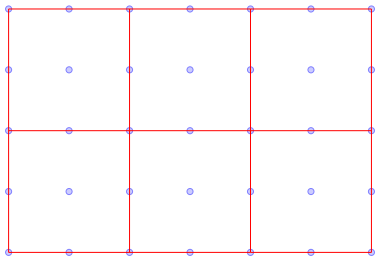} }
 \hspace{.6in}
 \subfigure[All quadrature points correspond to a finite difference grid.]{\includegraphics[scale=0.8]{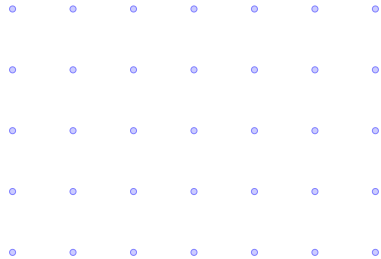}}
\caption{An illustration of  a uniform rectangular mesh for $Q^2$ elements and the $3\times3$ Gauss-Lobatto quadrature. }
\label{mesh}
 \end{figure}

\subsection{Variational formulation}
 Assume there is a function $\bar g \in H^1(\Omega)$ as a smooth extension of $g$ so that $\bar g|_{\partial \Omega} = g$.
 Introduce a bilinear form 
 \[
  B(\phi, \psi):= \langle \phi ,\psi \rangle+\langle\mathbf u\cdot \nabla \phi ,\psi\rangle+\mu \langle \nabla \phi ,\nabla \psi \rangle,
\]
where $\langle\cdot, \cdot\rangle$ is the standard $L^2$ inner product on $\Omega$, then
the  variational form of \eqref{ellipticeqn} is to find $\tilde \phi = \phi - \bar g \in H_0^1(\Omega)$ satisfying
\begin{equation}\label{nonhom-var}
   B(\tilde \phi, \psi)=\langle f,\psi\rangle -   B(\bar g,\psi) ,\quad \forall \psi\in H_0^1(\Omega).
 \end{equation}
 
 In practice, $\bar g$ is not used explicitly. By abusing notation,
the most convenient implementation is to consider
\[g(x,y)=\begin{cases}
   0,& \mbox{if}\quad (x,y)\in (0,1)\times(0,1),\\
   g(x,y),& \mbox{if}\quad (x,y)\in \partial\Omega,\\
  \end{cases}
\] 
and
$g_I\in V^h$ which is defined as the $Q^2$ Lagrange interpolation  of $g(x,y)$ at $3\times 3$ Gauss-Lobatto points for each rectangular cell on $  \Omega$.
Namely, $g_I\in V^h$ is the piecewise $Q^k$ interpolation of $g$ along the boundary grid points and $g_I=0$ at the interior grid points. 

The spectral element method, i.e., finite element method with suitable quadrature, is to  find $\phi_h \in V_0^h$, s.t.
\begin{equation}\label{nonhom-var-num3}
 B_h( \phi_h+g_I, \psi_h)=\langle f,v_h \rangle_h,\quad \forall \psi_h\in V_0^h,
\end{equation}
where $ B_h(\phi_h, \psi_h)$  and $\langle f,v_h\rangle_h$ denote using  $3\times3$ Gauss-Lobatto quadrature for integrals $B(\phi_h,\psi_h)$ and $\langle f,\psi_h\rangle $ respectively. 
 Then $\phi_h + g_I$ will be our numerical solution to $\phi(x,y)$ for  \eqref{ellipticeqn}.
Notice that \eqref{nonhom-var-num3} is not a straightforward approximation to \eqref{nonhom-var} since $\bar g$ is never used. 
The scheme \eqref{nonhom-var-num3}
 is fourth order accurate, see  \cite{li2019fourth, li2021accuracy}.
\subsection{One-dimensional fourth order scheme}
To derive an explicit expression of the scheme \eqref{nonhom-var-num3},
we start with a  one-dimensional steady state 
 equation on $x\in(0,1)$ with homogeneous Neumann boundary conditions:
\[\phi(x)+u(x) \phi'(x)-(\mu\phi')'=f(x),\quad \phi'(0)=\phi'(1)=0.\]
The variational form is find $\phi(x) \in H^1([0,1])$ satisfying
\[   B(\phi, \psi):=\langle \phi,\psi\rangle+\langle u \phi', \psi\rangle+\langle\mu \phi', \psi'\rangle= \langle f,\psi\rangle,\quad \forall \psi(x) \in H^1([0,1]),\]
where $\langle\cdot, \cdot\rangle$ denotes standard $L^2$ inner product. 
Consider a uniform mesh $x_i = ih$, $i = 0,1,\dots, n+1 $, $h=\frac{1}{n+1}$. Assume $n$ is odd and let $N=\frac{n+1}{2}$. Define  a finite element mesh for $P^2$ basis with intervals $I_k =[x_{2k},x_{2k+2}]$ for $k=0,\dots,N-1$. Define 
$$V^h=\{\psi\in C^0([0,1]): \psi|_{I_k}\in P^2(I_k), k = 0,\dots, N-1\}.$$ 
Let $\{\psi_i\}_{i=0}^{n+1} 	\subset V^h $ be a basis of $V^h$ such that $\psi_i(x_j)= \delta_{ij}, \,i,j=0,1,\dots,n+1$.
Then the continuous $P^2$ finite element method with $3$-point Gauss-Lobatto quadrature  is to seek $\phi_h(x) \in V^h$ satisfying  
\begin{align}
\label{1D-scheme-neumann}
 B_h(\phi_h, \psi_h):=\langle  \phi_h,\psi_i\rangle_h+\langle  u \phi_h', \psi_i\rangle_h+\langle\mu  \phi_h', \psi'_i\rangle_h= \langle f,\psi_i\rangle_h, \quad i=0,1,\dots,n+1,
\end{align}
where $\langle \cdot,\cdot\rangle_h$ denotes $3$-point Gauss-Lobatto quadrature for approximating integration on each interval $I_k$. 
Let $\phi_j=\phi_h(x_j)$ and $u_j=u(x_j)$,  then 
$\phi_h(x,t)=\sum\limits_{j=0}^{n+1} \phi_j \psi_j(x)$. We have 
$$\sum_{j=0}^{n+1} \phi_j \left( \langle  \psi_j,\psi_i\rangle_h+ \langle u \psi_j',\psi_i\rangle_h+\mu\langle \psi_j',\psi_i'\rangle_h\right) =  \sum_{j=0}^{n+1}  f_j \langle  \psi_j,\psi_i\rangle_h, \quad i=0,1,\dots,n+1.$$
The matrix form of this scheme is $\bar M \bar{\mathbf \phi}+\bar  U\bar T\bar{ \mathbf \phi}+\mu \bar S\bar{ \mathbf \phi}=\bar M \bar{\mathbf f}$,
where \[
\bar U=\begin{bmatrix}
u_0 &&&& \\
&u_1&&&\\
 &&\ddots&&\\
 &&&u_{n}&\\
 &&&&u_{n+1}
\end{bmatrix},
\bar{\phi}=\begin{bmatrix}
\phi_0\\ \phi_1\\ \vdots \\ \phi_{n}\\  \phi_{n+1}
\end{bmatrix},
\quad \bar{\textbf{f}}=\begin{bmatrix}
f_0\\f_1\\ \vdots \\f_{n} \\f_{n+1}
\end{bmatrix}.
\]
The stiffness matrix $\bar S$ has size $(n+2)\times(n+2)$ with $(i,j)$-th entry as $\langle  \psi_i',\psi_j'\rangle_h$, 
the  matrix $\bar T$  has size $(n+2)\times(n+2)$ with $(i,j)$-th entry as $\langle  \psi_j',\psi_i\rangle_h$
and
the lumped mass matrix $\bar M$ is a $(n+2)\times(n+2)$ diagonal matrix
with diagonal entries $h \begin{pmatrix}
\frac13,\frac43,\frac23,\frac43,\frac23,\dots,\frac23,\frac43,\frac13
\end{pmatrix}$.

Notice that $\bar U$ and $\bar M$ are diagonal thus they commute. 
Multiplying $\bar{M}^{-1}$ for both sides, we get a  finite difference representation
$$ \bar{  \phi}+ \bar U \bar D_1 \bar{ \phi}-\mu \bar D_2 \bar{\phi}=\bar{\mathbf f},$$
with square difference matrices for approximating first order and second order derivatives as
\[ \bar D_1=\bar{M}^{-1}\bar{T}=\frac{1}{2h}\left( \begin{smallmatrix}
-3 & 4 & -1 & & & & & & \\
-1 & 0 & 1 & & & & & & \\
\frac12 & -2 & 0 & 2 & -\frac12 & & &  & &\\
& & -1 & 0 & 1 & & &   \\
& & \frac12 & -2 & 0 & 2 & -\frac12 & & \\
& & & & -1 & 0 & 1 &\\
& & & & & \ddots & \ddots & \ddots &  & \\
& & &  & & & -1 & 0 & 1 &\\
& & & & & &  \frac12 & -2 & 0 & 2 & -\frac12\\
& & & & & & & & -1 & 0 & 1\\
& & & & & & & & 1 & -4 & 3
\end{smallmatrix}\right),\]
\[\bar D_2=-\bar{M}^{-1}\bar{S}=-\frac{1}{h^2}\left(\begin{smallmatrix}
\frac72& -4& \frac12 & & & & &&\\
  -1& 2& -1 & & & & &&\\
   \frac14 &-2& \frac72 &-2 & \frac14 & & &&\\
    & &  -1 & 2& -1 & & &&\\
  & & \frac14 &-2& \frac72 &-2 & \frac14 & &\\
 & & &   &  -1 & 2& -1 & &\\
 & & & & &\ddots &\ddots &\ddots &\\
&& && & \frac14 &-2& \frac72 &-2&\frac14\\
 & & &  & & & &-1 & 2  &-1\\
   & & &  & & & &\frac12 & -4  &\frac72
  \end{smallmatrix}\right).\]

Now consider the one-dimensional Dirichlet boundary value problem:
\[
\phi(x)+u(x) \phi'(x)-(\mu\phi')'=f(x)\textrm{ on } [0,1], \quad \phi(0) = \sigma_1,  \quad \phi(1) =  \sigma_2.
\]
Consider the same mesh as above and define 
$$V^h_0=\{\psi\in C^0([0,1]): \psi|_{I_k}\in P^2(I_k), k = 0,\dots, N-1; \psi(0)=\psi(1)=0\}.$$ Then $\{\psi_i\}_{i=1}^{n} 	\subset V^h $ is a basis of $V^h_0$.
Let $g_I(x)=\sigma_0 \psi_0(x)+\sigma_1\psi_{n+1}(x)$, then 
the one-dimensional version of \eqref{nonhom-var-num3} is to seek $\phi_h\in V^h_0$ satisfying  
\begin{equation}
 B_h  (\phi_h+g_I,\psi_i)_h=\langle f,\psi_i\rangle_h, \quad i=1,2,\dots,n,
\label{1D-scheme-Dirichlet}
\end{equation}
Notice that we can obtain \eqref{1D-scheme-Dirichlet} by simply setting $\phi_h(0)=\sigma_0$ and $\phi_h(1)=\sigma_1$ in \eqref{1D-scheme-neumann}. So the finite difference implementation of \eqref{1D-scheme-Dirichlet} is given as 
\begin{equation}
\begin{split}
 &{  \phi}+  U  D_1 \bar{ \phi}-\mu  D_2 \bar{\phi}= { f},\\
 &\phi_0=\sigma_0, \phi_1=\sigma_1,
 \end{split}
 \label{operatorform-1d}
 \end{equation}
with
 \[
U=\begin{bmatrix}
u_1 && \\
 &\ddots&\\
 &&u_{n}
\end{bmatrix},\quad  f=\begin{bmatrix}
f_1\\ \vdots  \\f_{n}
\end{bmatrix},
{\phi}=\begin{bmatrix}
\phi_1 \\ \vdots  \\  \phi_{n}
\end{bmatrix},\bar{\phi}=\begin{bmatrix}
\phi_0 \\ \phi_1\\ \vdots \\ \phi_{n}\\  \phi_1\end{bmatrix},
\]
and difference matrices of size $n\times(n+2)$:
\[  \resizebox{\textwidth}{!} 
{$D_1=\frac{1}{2h}\left( \begin{smallmatrix}
-1 & 0 & 1 & & & & & & \\
\frac12 & -2 & 0 & 2 & -\frac12 & & &  & &\\
& & -1 & 0 & 1 & & &   \\
& & \frac12 & -2 & 0 & 2 & -\frac12 & & \\
& & & & -1 & 0 & 1 &\\
& & & & & \ddots & \ddots & \ddots &  & \\
& & &  & & & -1 & 0 & 1 &\\
& & & & & &  \frac12 & -2 & 0 & 2 & -\frac12\\
& & & & & & & & -1 & 0 & 1\\
\end{smallmatrix}\right), D_2=-\frac{1}{h^2}\left(\begin{smallmatrix}
  -1& 2& -1 & & & & &&\\
   \frac14 &-2& \frac72 &-2 & \frac14 & & &&\\
    & &  -1 & 2& -1 & & &&\\
  & & \frac14 &-2& \frac72 &-2 & \frac14 & &\\
 & & &   &  -1 & 2& -1 & &\\
 & & & & &\ddots &\ddots &\ddots &\\
&& && & \frac14 &-2& \frac72 &-2&\frac14\\
 & & &  & & & &-1 & 2  &-1\\
  \end{smallmatrix}\right).$}\]
Let $I_n$ denote the identity matrix of size $n\times n$, and define a restriction matrix 
\[R=\begin{pmatrix}
\mathbf 0 & I_n & \mathbf 0\\
\end{pmatrix}_{n\times(n+2)}.\] 
  Then the left hand size of \eqref{operatorform-1d} for interior points can be regarded as 
  a linear operator on $\bar \phi$: 
  $\mathcal L(\bar \phi)=R\bar\phi +U  D_1 \bar{ \phi}-\mu  D_2 \bar{\phi}=f$.
  
The scheme can also be explicitly written as
\begin{subequations}
\label{explicitscheme-1d}
  \begin{align}
  &  \phi_0=\sigma_0, \quad 
  \phi_{n+1}=\sigma_1;\\
 &   \phi_i+u_i\frac{\phi_{i+1}-\phi_{i-1}}{2h}+\mu\frac{-\phi_{i-1}+2\phi_i-\phi_{i+1}}{h^2}=f_i,  \text{if $i$ is odd, i.e., $x_i$ is a cell center};\\
 &  \phi_i+u_i\frac{\phi_{i-2}-4\phi_{i-1}+4\phi_{i+1}-\phi_{i+2}}{4h}
+\mu\frac{\phi_{i-2}-8\phi_{i-1}+14\phi_i-8\phi_{i+1}+\phi_{i+2}}{4h^2}=f_i,\\
 & \text{if $i$ is even, i.e., $x_i$ is a cell end}.\notag
  \end{align}
\end{subequations}

  \begin{rmk}
  The difference matrices $D_1$ and $D_2$ are only second order accurate in truncation errors approximating derivatives. But they give a fourth order accurate scheme for 
  second order PDEs such as elliptic equations
  \cite{li2019fourth}, and wave and parabolic equations \cite{li2021accuracy}. 
However,  if only using $D_1$ for a pure convection equation, then the scheme can only be second order accurate. 
  \end{rmk}

\subsection{Two-dimensional fourth order scheme}
Consider a uniform grid $(x_i,y_j)$ for a rectangular domain $\bar \Omega=[0,1]\times[0,1]$
where
$x_i = i\Delta x$, $i = 0,1,\dots, n_x+1$, $\Delta x=\frac{1}{n_x+1}$
and
$y_j = j\Delta y$, $j = 0,1,\dots, n_y+1$, $\Delta y=\frac{1}{n_y+1}$. 
Assume $n_x$ and $n_y$ are odd and let $N_x=\frac{n_x+1}{2}$ and 
$N_y=\frac{n_y+1}{2}$. We consider  a  $Q^2$ finite element mesh   consisting of rectangular cells $e_{kl} =[x_{2k},x_{2k+2}]\times
[y_{2l},y_{2l+2}]$ for $k=0,\dots,N_x-1$ and $l=0,\dots,N_y-1$. 
For given functions $u, v, f$ and $g$, let indices denote point values at corresponding grid points, e.g., $u_{ij}=u(x_i, y_j)$.
Let $u, v, f$ denote matrices of size $n_y\times n_x$ consisting of point values of corresponding functions, e.g., 
\begin{align*}
   u=\begin{pmatrix}
   u_{11} & u_{12} & \dots & u_{1,n_x}\\
   u_{21} & u_{22} & \dots & u_{2,n_x}\\
   \vdots & \vdots & & \vdots \\
      u_{n_y,1} & u_{n_y,2} & \dots & u_{n_y,n_x}
   \end{pmatrix}_{n_y\times n_x},  f=\begin{pmatrix}
   f_{11} & f_{12} & \dots & f_{1,n_x}\\
   f_{21} & f_{22} & \dots & f_{2,n_x}\\
   \vdots & \vdots & & \vdots \\
      f_{n_y,1} & f_{n_y,2} & \dots & f_{n_y,n_x}
   \end{pmatrix}_{n_y\times n_x}. 
   \end{align*}
Let $D_{ix}$ and $D_{iy}$ denote the $D_i$ ($i=1,2$) matrices for $x$ and $y$ variables correspondingly, e.g.,
\begin{align*}D_{1x}=\frac{1}{2\Delta x}\left( \begin{smallmatrix}
-1 & 0 & 1 & & & & & & \\
\frac12 & -2 & 0 & 2 & -\frac12 & & &  & &\\
& & -1 & 0 & 1 & & &   \\
& & \frac12 & -2 & 0 & 2 & -\frac12 & & \\
& & & & -1 & 0 & 1 &\\
& & & & & \ddots & \ddots & \ddots &  & \\
& & &  & & & -1 & 0 & 1 &\\
& & & & & &  \frac12 & -2 & 0 & 2 & -\frac12\\
& & & & & & & & -1 & 0 & 1\\
\end{smallmatrix}\right)_{n_x\times (n_x+2)}, \\
D_{2y}=-\frac{1}{\Delta y^2}\left(\begin{smallmatrix}
  -1& 2& -1 & & & & &&\\
   \frac14 &-2& \frac72 &-2 & \frac14 & & &&\\
    & &  -1 & 2& -1 & & &&\\
  & & \frac14 &-2& \frac72 &-2 & \frac14 & &\\
 & & &   &  -1 & 2& -1 & &\\
 & & & & &\ddots &\ddots &\ddots &\\
&& && & \frac14 &-2& \frac72 &-2&\frac14\\
 & & &  & & & &-1 & 2  &-1\\
  \end{smallmatrix}\right)_{n_y\times (n_y+2)}.
   \end{align*}

 Let  $ \phi$ be a $n_y\times n_x$ matrix consisting of interior point values of $\phi_h$:
 \[  \phi=\begin{pmatrix}
   \phi_{11} & \dots&\phi_{1,n_x} \\
   \vdots & &\vdots  \\
   \phi_{n_y,1} & \dots & \phi_{n_y,n_x} 
   \end{pmatrix}_{n_y\times n_x}.\]  
 Let  
   $\bar \phi$ be a $(n_y+2)\times(n_x+2)$ matrix consisting of both interior and boundary point values $\phi_h+g_I$:
 \[\bar \phi=\begin{pmatrix}
   \phi_{00} & \phi_{01} & \dots & \phi_{0,n_x}& \phi_{0,n_x+1}\\
   \phi_{10} & \phi_{11} & \dots&\phi_{1,n_x} & \phi_{1,n_x+1}\\
   \vdots & \vdots & &\vdots& \vdots \\
         \phi_{n_y,0} & \phi_{n_y,1} & \dots & \phi_{n_y,n_x} &  \phi_{n_y,n_x+1}\\
      \phi_{n_y+1,0} & \phi_{n_y+1,2} & \dots & \phi_{n_y+1,n_x} &  \phi_{n_y+1,n_x+1}
   \end{pmatrix}_{n_y\times n_x}.\]

 Let $R_x$ and $R_y$ denote restriction matrices:
\[R_x=\begin{pmatrix}
\mathbf 0 & I_{n_x} & \mathbf 0\\
\end{pmatrix}_{n_x\times(n_x+2)},\quad R_y=\begin{pmatrix}
\mathbf 0 & I_{n_y} & \mathbf 0\\
\end{pmatrix}_{n_y\times(n_y+2)}.\]   
Then the scheme  \eqref{nonhom-var-num3} for interior grid points  is equivalent to 
the linear operator form
 \begin{equation}  {\mathcal L}(\bar\phi):=\phi+u.*(R_y \bar \phi D_{1x}^T)+v.*(D_{1y} \bar \phi R_x^T)-\mu(R_y \bar \phi D_{2x}^T+ D_{2y} \bar \phi R_x^T)=f,
  \label{operatorform-2d}
 \end{equation} where $.*$ denotes entrywise product of two matrices.
 For the boundary points, we simply have
 \[\phi_{ij}=g_{ij},\quad \text{if}\,\, (x_i, y_j)\in \partial\Omega.\]
 
Define the following operators:
\begin{itemize}
 \item $\otimes$ denotes Kronecker product of two matrices;
 \item $vec(X)$ denotes  the vectorization of the matrix $X$ by rearranging $X$ into a vector column by column;
\item $diag(\mathbf x)$ denote a diagonal matrix with the vector $\mathbf x$ as diagonal entries.
\end{itemize}
Then \eqref{operatorform-2d} is also equivalent to an abstract matrix-vector form
\[ \resizebox{\textwidth}{!} 
{$\left[R_x\otimes R_x+diag(vec(u))D_{1x}\otimes R_y+diag(vec(v))R_x\otimes D_{1y}-\mu (D_{2x}\otimes R_y+ R_x\otimes D_{2y})\right] vec(\bar\phi)=vec(f).
$} \]

\begin{figure}[htbp]
\begin{center}
 \scalebox{0.7}{
\begin{tikzpicture}[samples=100, domain=-3:3, place/.style={circle,draw=blue!50,fill=blue,thick,
inner sep=0pt,minimum size=1.5mm},transition/.style={circle,draw=red,fill=red,thick,inner sep=0pt,minimum size=2mm}
,point/.style={circle,draw=black,fill=black,thick,inner sep=0pt,minimum size=2mm}]

\draw[color=red] (-2,-2)--(-2,2);
\draw[color=red] (-2,-2)--(2,-2);
\draw[color=red] (2,-2)--(2, 2);
\draw[color=red] (-2,2)--(2, 2);
\node at ( -2,-2) [place] {};
\node at ( 0,-2) [point] {};
\node at ( 2,-2) [place] {};

\node at ( -2,0) [point] {};
\node at ( 0,0) [transition] {};
\node at ( 2,0) [point] {};

\node at ( -2,2) [place] {};
\node at ( 0,2) [point] {};
\node at ( 2,2) [place] {};
 \end{tikzpicture}
}
\caption{Three types of interior grid points: red cell center, blue knots and black edge centers for a finite element cell. }
\end{center}
\label{fig-points}
\end{figure}

For interior grid points, there are three types: cell center, edge center and knots. See Figure \ref{fig-points}. The scheme can also be explicitly written as:
\begin{align*}
&\phi_{ij}+\Delta t u_{ij}\frac{\phi_{i+1,j}-\phi_{i-1,j}}{2\Delta x}+\Delta t v_{ij}\frac{\phi_{i,j+1}-\phi_{i,j-1}}{2\Delta y}\\
 &+\Delta t \mu\frac{-\phi_{i-1,j}+2\phi_{ij}-\phi_{i+1,j}}{\Delta x^2}
  +\Delta t \mu\frac{-\phi_{i,j-1}+2\phi_{ij}-\phi_{i,j+1}}{\Delta y^2}=f_{ij}, \quad \text{if $(x_i, y_j)$ is a cell center;}\\
&\phi_{ij}+  \Delta t u_{ij}\frac{\phi_{i+1,j}-\phi_{i-1,j}}{2\Delta x}+ \Delta t v_{ij}\frac{\phi_{i,j-2}-4\phi_{i,j-1}+4\phi_{i,j+1}-\phi_{i,j+2}}{4\Delta y}\\
&+\Delta t \mu\frac{-\phi_{i-1,j}+2\phi_{i,j}-\phi_{i+1,j}}{\Delta x^2}+\Delta t \mu\frac{\phi_{i,j-2}-8\phi_{i,j-1}+14\phi_{i,j}-8\phi_{i,j+1}+\phi_{i,j+2}}{4\Delta y^2}=f_{ij}\\
&\text{if $(x_i, y_j)$ is an interior edge center for an edge parallel to x-axis};\\
&\phi_{ij}+   \Delta t u_{ij}\frac{\phi_{i-2,j}-4\phi_{i-1,j}+4\phi_{i+1,j}-\phi_{i+2,j}}{4\Delta x}+\Delta t v_{ij}\frac{\phi_{i,j+1}-\phi_{i,j-1}}{2\Delta y}\\
&\Delta t \mu\frac{\phi_{i-2,j}-8\phi_{i-1,j}+14\phi_{i,j}-8\phi_{i+1,j}+\phi_{i+2,j}}{4\Delta x^2}+\Delta t \mu\frac{-\phi_{i,j-1}+2\phi_{i,j}-\phi_{i,j+1}}{\Delta y^2}=f_{ij},\\
&\text{if $(x_i, y_j)$ is an interior edge center for an edge parallel to y-axis};\\
&\phi_{ij}+   \Delta t u_{ij}\frac{\phi_{i-2,j}-4\phi_{i-1,j}+4\phi_{i+1,j}-\phi_{i+2,j}}{4\Delta x}+\Delta t v_{ij}\frac{\phi_{i,j-2}-4\phi_{i,j-1}+4\phi_{i,j+1}-\phi_{i,j+2}}{4\Delta y}\\\
&\resizebox{\hsize}{!}{$ +\Delta t \mu\frac{\phi_{i-2,j}-8\phi_{i-1,j}+14\phi_{i,j}-8\phi_{i+1,j}+\phi_{i+2,j}}{4\Delta x^2}+\Delta t \mu\frac{\phi_{i,j-2}-8\phi_{i,j-1}+14\phi_{i,j}-8\phi_{i,j+1}+\phi_{i,j+2}}{4\Delta y^2}=f_{ij} $}\\
&\text{if $(x_i, y_j)$ is an interior knot}.
\end{align*}

\subsection{The second order scheme}

 If using $P^1$ basis   for one-dimensional case or  $Q^1$ basis for two-dimensional case in continuous finite element method with $2$-point Gauss-Lobatto quadrature
 in \eqref{nonhom-var-num3}, 
  we get exactly the classical second order centered difference scheme
which can be written in the same abstract form \eqref{operatorform-1d} or \eqref{operatorform-2d} with difference matrices defined as 
  \[  D_1=\frac{1}{2h}\left( \begin{smallmatrix}
-1 & 0 & 1 & &  \\
 & -1 & 0 & 1 & &    \\
 & & \ddots & \ddots & \ddots &   \\
 &  & & -1 & 0 & 1 
 \end{smallmatrix}\right), D_2=-\frac{1}{h^2}\left(\begin{smallmatrix}
  -1& 2& -1 & & & \\
    &  -1 & 2& -1 & &\\
  & &\ddots &\ddots &\ddots &\\
 & & &-1 & 2  &-1
  \end{smallmatrix}\right).\]
 The scheme can also be written explicitly in one dimension as 
   \[\phi_i+u_i\frac{\phi_{i+1}-\phi_{i-1}}{2h}+\mu\frac{-\phi_{i-1}+2\phi_i-\phi_{i+1}}{h^2}=f_i,\]
and in two dimensions as
  \begin{equation}
  \resizebox{\hsize}{!}{$ \label{2ndscheme}\phi_{ij}+u_{ij}\frac{\phi_{i+1,j}-\phi_{i-1,j}}{2\Delta x}+v_{ij}\frac{\phi_{i,j+1}-\phi_{i,j-1}}{2\Delta y}+\mu\frac{-\phi_{i-1,j}+2\phi_{ij}-\phi_{i+1,j}}{\Delta x^2}
  +\mu\frac{-\phi_{i,j-1}+2\phi_{ij}-\phi_{i,j+1}}{\Delta y^2}=f_i.$}
  \end{equation}

\section{Monotonicity and discrete maximum principle}
\label{sec-mono}
\subsection{Backward Euler time discretization}
Now consider backward Euler time discretization 
for solving an initial value problem for a linear convection-diffusion equation  on a square domain $ \Omega=(0,1)\times(0,1)$ with
Dirichlet boundary conditions:
\begin{align*}
&\phi_t+u\phi_x+v \phi_y=\nabla\cdot(\mu \nabla \phi), \quad (x,y)\in \Omega,\\
& \phi(x,y,0)=f(x,y),\quad (x,y)\in  \Omega,\\
&\phi(x,y,t)=g(x,y,t),\quad (x,y)\in\partial \Omega.
\end{align*}
We get an elliptic equation for $\phi^{n+1}$:
\begin{equation}
 \phi^{n+1}+\Delta t u^{n+1}\phi^{n+1}_x+\Delta t v ^{n+1}\phi^{n+1}_y-\Delta t \nabla\cdot(\mu \nabla \phi^{n+1})=\phi^n, 
 \label{backwardeuler}
  \end{equation}
with boundary condition $\phi(x,y,t_{n+1})=g(x,y,t_{n+1}), (x,y)\in\partial \Omega$.
With the same notations in Section \ref{sec-scheme},  the variational difference scheme for \eqref{backwardeuler} at interior grid points  is
\begin{subequations}
  \label{fullscheme}
 \begin{equation}
\resizebox{\textwidth}{!} 
{$  {\mathcal L}(\bar\phi^{n+1}):=\phi^{n+1}+\Delta t \left[u^{n+1}.*(R_y \bar \phi^{n+1} D_{1x}^T)+v^{n+1}.*(D_{1y} \bar \phi^{n+1} R_x^T)-\mu(R_y \bar \phi^{n+1} D_{2x}^T+ D_{2y} \bar \phi^{n+1} R_x^T)\right]=\phi^{n}.
$}
  \end{equation}
Now define a linear operator $\bar{\mathcal L}:\mathbbm R^{(n_y+2)\times(n_x+2)}\rightarrow \mathbbm R^{(n_y+2)\times(n_x+2)}$:
\[ \bar{\mathcal L} (\bar{\phi}^{n+1})_{ij}:=\begin{cases} {\mathcal L}(\bar{\phi}^{n+1})_{ij}, & \quad (x_i, y_j)\in\Omega,\\
{\phi}^{n+1}_{ij}, & \quad (x_i, y_j)\in\partial\Omega.
 \end{cases}\]
 Then the finite difference scheme can be written as 
\begin{equation}
\label{scheme-abstract}
 \begin{cases}\bar{\mathcal L} (\bar{\phi}^{n+1})_{ij} =\phi^n_{ij}, & \quad (x_i, y_j)\in\Omega,\\
\bar{\mathcal L} (\bar{\phi}^{n+1})_{ij}=g_{ij}^{n+1}, & \quad (x_i, y_j)\in\partial\Omega.
 \end{cases}\end{equation}
\end{subequations}

 We first have two straightforward results:   
 
 \begin{theorem}
 \label{thm-zerosum}
 Let $\bar{\mathbf 1}$ denote a matrix of the same size as $\bar \phi$, for the scheme operator in \eqref{fullscheme}, $\bar{\mathcal L}(\bar{\mathbf 1})=\bar{\mathbf 1}.$
 \end{theorem}
\begin{proof}
Notice that row sums of $D_1$ and $D_2$ in second order and fourth order accurate schemes are all zeros, thus $D_1 \mathbf 1=D_2 \mathbf 1=\mathbf 0$, which implies the result.
\end{proof}

\begin{theorem}
\label{thm-dmp}
For the scheme \eqref{scheme-abstract}, let $\bar L$ be the matrix representation of the linear operator $\bar{\mathcal L}$. 
If the inverse matrix has non-negative entries, i.e., $\bar L^{-1}\geq 0$, then the finite difference scheme satisfies a discrete maximum principle:
\begin{equation}
\label{dmp}
\min\left\{ \min_{(x_i, y_j)\in \Omega} \phi^{n}_{ij}, \min_{(x_i, y_j)\in \partial\Omega} g^{n+1}_{ij}\right \} \leq \phi^{n+1}_{ij}\leq \max\left \{ \max_{(x_i, y_j)\in \Omega} \phi^{n}_{ij}, 
\max_{(x_i, y_j)\in \partial\Omega} g^{n+1}_{ij}\right \}.
\end{equation}
\end{theorem}            
\begin{proof}

Let $\mathbf 1$ denote the vector consisting of ones, then Theorem \ref{thm-zerosum} implies $\bar L \mathbf 1=\mathbf 1$ thus 
$\bar L^{-1} \mathbf 1=\mathbf 1$. Since all entries in $\bar L^{-1}$ are non-negative, each row in $\bar L^{-1}$ forms a set of coefficients for a convex combination, which implies the 
maximum principle.
\end{proof}

\subsection{M-matrix and the second order scheme}

 Nonsingular M-matrices are inverse-positive matrices, which is the main tool for proving inverse positivity. There are many equivalent definitions or characterizations of M-matrices, see 
\cite{plemmons1977m}. 
 One convenient sufficient but not necessary characterization of nonsingular M-matrices is as follows:
\begin{theorem}
\label{rowsumcondition-thm}
For a real square matrix $A$  with positive diagonal entries and non-positive off-diagonal entries, $A$ is a nonsingular M-matrix if  all the row sums of $A$ are non-negative and at least one row sum is positive. 
\end{theorem}
\begin{proof}
By condition $C_{10}$ in \cite{plemmons1977m}, $A$ is a nonsingular M-matrix if and only if $A+a I$ is nonsingular for any $a\geq 0$.  
Since all the row sums of $A$ are non-negative and at least one row sum is positive,
the matrix $A$ is  
irreducibly diagonally dominant thus nonsingular, and $A+a I$ is strictly diagonally dominant thus nonsingular for any $a>0.$
\end{proof}
 
 By condition $K_{35}$ in \cite{plemmons1977m}, a sufficient and necessary characterization of nonsingular M-matrices is the following:
 \begin{theorem}
\label{rowsumcondition-thm2}
 For a real square matrix $A$  with positive diagonal entries and non-positive off-diagonal entries, $A$ is a nonsingular M-matrix if  and only if that there exists a positive diagonal matrix 
 $D$ such that $AD$ has all
positive row sums.
 \end{theorem}
 
If using second order scheme \eqref{2ndscheme} in  \eqref{scheme-abstract}, then the scheme operator $\bar{\mathcal L}$ acting on   $\bar \phi$ is given as
\[
 \begin{cases}\bar{\mathcal L} (\bar{\phi})_{ij} =&\phi_{ij}+\Delta t u_{ij}\frac{\phi_{i+1,j}-\phi_{i-1,j}}{2\Delta x}+\Delta t v_{ij}\frac{\phi_{i,j+1}-\phi_{i,j-1}}{2\Delta y}\\
 &+\Delta t \mu\frac{-\phi_{i-1,j}+2\phi_{ij}-\phi_{i+1,j}}{\Delta x^2}
  +\Delta t \mu\frac{-\phi_{i,j-1}+2\phi_{ij}-\phi_{i,j+1}}{\Delta y^2}, \qquad (x_i, y_j)\in\Omega,\\
\bar{\mathcal L} (\bar{\phi})_{ij}=&\phi_{ij},  \qquad (x_i, y_j)\in\partial\Omega.
\end{cases}
\]
 For interior points $(x_i, y_j)\in\Omega$, we have 
\[ \resizebox{\textwidth}{!} 
{$\bar{\mathcal L} (\bar{\phi})_{ij} =\left(1+\frac{2\mu \Delta t}{\Delta x^2}+\frac{2\mu \Delta t}{\Delta y^2}\right)\phi_{ij}-\frac{\Delta t}{\Delta x}\left(\frac{\mu}{\Delta x} -\frac{u_{ij}}{2}\right)(\phi_{i+1,j}+\phi_{i-1,j})-\frac{\Delta t}{\Delta y}\left(\frac{\mu}{\Delta y} -\frac{v_{ij}}{2}\right)(\phi_{i,j+1}+\phi_{i,j-1})$}.\]
 Assume  
 \begin{equation}
\label{meshconstraint1}
 \Delta x\max_{ij}|u_{ij}|\leq 2\mu,\quad  \Delta y\max_{ij}|v_{ij}|\leq 2\mu,
 \end{equation}
 then all off-diagonal entries of $\bar L$ will be non-positive, thus $\bar L$ is an M-matrix. 
 \begin{theorem}
 \label{thm-monotonicity1}
 Under the mesh constraints \eqref{meshconstraint1}, the second order accurate scheme \eqref{2ndscheme} is monotone, i.e., $\bar L^{-1}\geq 0$, and satisfies the discrete maximum principle \eqref{dmp}. 
 \end{theorem}

 \subsection{ Lorenz's condition for the fourth order scheme}
Unfortunately, almost all high order schemes will lead to positive off-diagonal entries in the system matrix, which can no longer be an M-matrix, see \cite{cross2020monotonicity}. 
In   \cite{lorenz1977inversmonotonie} Lorenz proposed a convenient condition under which a matrix can be shown to be a product of M-matrices.
We briefly review the Lorenz's condition in this subsection. 

\begin{defi}
Let $\mathcal N = \{1,2,\dots,n\}$. For $\mathcal N_1, \mathcal N_2 \subset \mathcal N$, we say a matrix $A$ of size $n\times n$ connects $\mathcal N_1$ with $\mathcal N_2$ if 
\begin{equation}
\forall i_0 \in \mathcal N_1, \exists i_r\in \mathcal N_2, \exists i_1,\dots,i_{r-1}\in \mathcal N \quad \mbox{s.t.}\quad  a_{i_{k-1}i_k}\neq 0,\quad k=1,\cdots,r.
\label{condition-connect}
\end{equation}
If perceiving $A$ as a directed graph adjacency matrix of vertices labeled by $\mathcal N$, then \eqref{condition-connect} simply means that there exists a directed path from any vertex in $\mathcal N_1$ to at least one vertex in $\mathcal N_2$.  
In particular, if $\mathcal N_1=\emptyset$, then any matrix $A$  connects $\mathcal N_1$ with $\mathcal N_2$.
\end{defi}

Given a square matrix $A$ and a column vector $\mathbf x$, we define
\[\mathcal N^0(A\mathbf x)=\{i: (A\mathbf x)_i=0\},\quad 
\mathcal N^+(A\mathbf x)=\{i: (A\mathbf x)_i>0\}.\]

Given a matrix $A=[a_{ij}]\in \mathbbm{R}^{n\times n}$, define its diagonal, positive and negative off-diagonal parts as $n\times n$ matrices $A_d$, $A_a$, $A_a^+$, $A_a^-$:
\[(A_d)_{ij}=\begin{cases}
a_{ii}, & \mbox{if} \quad i=j\\
0, & \mbox{if} \quad  i\neq j
\end{cases}, \quad A_a=A-A_d,
\]
\[(A_a^+)_{ij}=\begin{cases}
a_{ij}, & \mbox{if} \quad a_{ij}>0,\quad i\neq j\\
0, & \mbox{otherwise}.
\end{cases}, \quad A_a^-=A_a-A^+_a.
\]

The following  result was proven in  \cite{lorenz1977inversmonotonie}. See also \cite{li2019monotonicity} for a detailed proof.
\begin{theorem}[Lorenz's condition] \label{thm3}
If $A^-_a$ has a decomposition: $A^-_a = A^z + A^s = (a_{ij}^z) + (a_{ij}^s)$ with $A^s\leq 0$ and $A^z \leq 0$, such that 
\begin{subequations}
 \label{lorenz-condition}
\begin{align}
& A_d + A^z \textrm{ is a nonsingular M-matrix},\label{cond1}\\ 
& A^+_a \leq A^zA^{-1}_dA^s \textrm{ or equivalently } \forall a_{ij} > 0 \textrm{ with } i \neq j, a_{ij} \leq \sum_{k=1}^n a_{ik}^za_{kk}^{-1}a_{kj}^s,\label{cond2}\\
& \exists \mathbf e \in \mathbbm{R}^n\setminus\{\mathbf 0\}, \mathbf e\geq 0 \textrm{ with $A\mathbf e \geq 0$ s.t. $A^z$ or $A^s$  connects $\mathcal N^0(A\mathbf e)$ with $\mathcal N^+(A\mathbf e)$.} \label{cond3}
\end{align}
\end{subequations}
Then $A$ is a product of two nonsingular M-matrices thus $A^{-1}\geq 0$.
\end{theorem}

In general, the condition \eqref{cond3} can be difficult to verify. But for the finite difference schemes, the vector $\mathbf e$ can be taken as $\mathbf 1$  to simply \eqref{cond3}. In particular, for the fourth order accurate scheme \eqref{fullscheme}, we have $\bar{\mathcal L}(\bar{\mathbf 1})=\bar{\mathbf 1}$ thus
$\bar L \mathbf 1=\mathbf 1$. Therefore, $\mathcal N^0(\bar L \mathbf 1)=\emptyset$ implies that  the condition \eqref{cond3} is trivially satisfied. 
So we can state a simpler Lorenz's condition for the scheme considered in this paper:
\begin{theorem}
 \label{newthm3}
Let $A$ denote the matrix representation of a new linear operator $\mathcal A:=\frac{h^2}{\mu\Delta t}\bar{\mathcal L}$ for the   scheme \eqref{fullscheme}, 
 with a corresponding matrix $A:=\frac{h^2}{\mu\Delta t}\bar L$.
Assume $A^-_a$ has a decomposition $A^-_a = A^z + A^s$ with $A^s\leq 0$ and $A^z \leq 0$.  
Then $A^{-1}\geq 0$ if the following are satisfied:
\begin{enumerate}
\item $A_d + A^z$ is a nonsingular M-matrix;
\item $A^+_a \leq A^zA^{-1}_dA^s$.
\end{enumerate}
\end{theorem}

\subsection{Verification of Lorenz's condition: one-dimensional case}

We first show how to verify Theorem \ref{newthm3} for the one-dimensional version of \eqref{fullscheme}.
For simplicity, let $c=\frac{h^2}{\mu\Delta t}$, then we have the linear operator $\mathcal A:=\frac{h^2}{\mu\Delta t}\bar{\mathcal L}=c\bar{\mathcal L}$ with a corresponding matrix $A:=c\bar L$. 
Applying \eqref{explicitscheme-1d} to \eqref{backwardeuler},  we get
the explicit expression  of $\mathcal A(\bar \phi)$ as
  \begin{align*}
  &  \mathcal A (\bar \phi)_0 =c\phi_0, \quad 
  \mathcal A (\bar \phi)_{n+1} =c\phi_{n+1}; \\
 &   \mathcal A (\bar \phi)_i= c \phi_i+\frac{hu_i}{2\mu}(\phi_{i+1}-\phi_{i-1})+(-\phi_{i-1}+2\phi_i-\phi_{i+1}), \quad \text{if   $x_i$ is a cell center};\\
 &  \mathcal A (\bar \phi)_i= c \phi_i+\frac{hu_i}{\mu}\frac{\phi_{i-2}-4\phi_{i-1}+4\phi_{i+1}-\phi_{i+2}}{4}
+\frac{\phi_{i-2}-8\phi_{i-1}+14\phi_i-8\phi_{i+1}+\phi_{i+2}}{4},\\
 & \text{if  $x_i$ is a cell end}.
  \end{align*}
  
  \subsubsection{A splitting $A_a^-=A^z+A^s$ }
In order to have fixed signs for all entries, we assume 
\begin{equation}
h|u_i|\leq 2\mu,\quad \forall i.
\label{constraint1}
\end{equation}
Then we have
  \begin{align*}
  &  \mathcal A_d (\bar \phi)_0 =c\phi_0, \quad 
  \mathcal A^d (\bar \phi)_{n+1} =c\phi_{n+1}, \\
 &   \mathcal A_d (\bar \phi)_i= c \phi_i+2\phi_i,   \quad \text{if $i$ is odd, i.e., $x_i$ is a cell center};\notag\\
 &  \mathcal A_d (\bar \phi)_i= c \phi_i+\frac72\phi_i,  \quad\text{if $i$ is even, i.e., $x_i$ is a cell end}.\notag\\
  \end{align*} 
  \begin{align*}
  &  \mathcal A_a^+ (\bar \phi)_0 =\mathcal A^+ (\bar \phi)_{n+1} =0, \\
 &   \mathcal A_a^+ (\bar \phi)_i= 0,   \quad \text{if $i$ is odd, i.e., $x_i$ is a cell center};\notag\\
 &  \mathcal A_a^+ (\bar \phi)_i= \left(1+\frac{hu_i}{\mu}\right)\frac14 \phi_{i-2}+\left(1-\frac{hu_i}{\mu}\right)\frac14 \phi_{i+2},  \quad\text{if $i$ is even, i.e., $x_i$ is a cell end}.\notag\\
  \end{align*}
  \begin{align*}
  &  \mathcal A_a^- (\bar \phi)_0 =\mathcal A^- (\bar \phi)_{n+1} =0, \\
 &   \mathcal A_a^- (\bar \phi)_i= -\left(1+\frac{h u_i}{2\mu}\right)\phi_{i-1}-\left(1-\frac{h u_i}{2\mu}\right)\phi_{i+1},   \quad \text{if $i$ is odd, i.e., $x_i$ is a cell center};\notag\\
 &  \mathcal A_a^- (\bar \phi)_i=  -2\left(1+\frac{h u_i}{2\mu}\right)\phi_{i-1}-2\left(1-\frac{h u_i}{2\mu}\right)\phi_{i+1}  \quad\text{if $i$ is even, i.e., $x_i$ is a cell end}.\notag\\
  \end{align*}
Next we define a splitting $A_a^-=A^z+A^s$ as:
  \begin{align*}
 &   \mathcal A^z (\bar \phi)_i=0,   \quad \text{if $i$ is odd, i.e., $x_i$ is a cell center};\notag\\
 &  \mathcal A^z (\bar \phi)_i=  -2\left(1-\frac{h u_i}{2\mu}\right)\phi_{i-1}-2\left(1+\frac{h u_i}{2\mu}\right)\phi_{i+1}  \quad\text{if $i$ is even, i.e., $x_i$ is a cell end}.\notag\\
  \end{align*}
  \begin{align*}
 &   \mathcal A^s (\bar \phi)_i= -\left(1-\frac{h u_i}{2\mu}\right)\phi_{i-1}-\left(1+\frac{h u_i}{2\mu}\right)\phi_{i+1},   \quad \text{if $i$ is odd, i.e., $x_i$ is a cell center};\notag\\
 &  \mathcal A^s (\bar \phi)_i=  0, \quad\text{if $i$ is even, i.e., $x_i$ is a cell end}.\notag\\
  \end{align*}
  \subsubsection{Verification of $A_d+A^z$ being an M-matrix}
For simplicity, define $B=A_d+A^z$ then the corresponding linear operator $\mathcal B=\mathcal A_d+\mathcal A^z$:
  \begin{align*}
   &  \mathcal B (\bar \phi)_0 =c\phi_0, \quad 
  \mathcal B (\bar \phi)_{n+1} =c\phi_{n+1}, \\
 &   \mathcal B (\bar \phi)_i=c \phi_i+2\phi_i,   \quad \text{if $i$ is odd, i.e., $x_i$ is a cell center};\notag\\
 &  \mathcal B (\bar \phi)_i= \left(c+\frac72\right)\phi_i -2\left(1-\frac{h u_i}{2\mu}\right)\phi_{i-1}-2\left(1+\frac{h u_i}{2\mu}\right)\phi_{i+1}, \\
 &  \quad\text{if $i$ is even, i.e., $x_i$ is a cell end}.\notag\\
  \end{align*}
  Since for even $i$, $\mathcal B(\mathbf 1)=c-\frac12$ which is for small $c$, thus Theorem \ref{rowsumcondition-thm} cannot be applied.

  Define the following linear operator $\mathcal D$:
  \begin{align*}
    &  \mathcal D (\bar \phi)_0 =\phi_0, \quad 
  \mathcal D (\bar \phi)_{n+1} =\phi_{n+1}, \\
 &   \mathcal D (\bar \phi)_i=\frac12\phi_i,   \quad \text{if $i$ is odd, i.e., $x_i$ is a cell center};\notag\\
 &  \mathcal D (\bar \phi)_i=  \phi_i,  \quad\text{if $i$ is even, i.e., $x_i$ is a cell end}.\notag
  \end{align*}
Let  $D$ be the matrix representing the operator $\mathcal D$ then $D$ is a diagonal matrix with positive diagonal entries. 
And we have  \begin{align*}
     \mathcal B [\mathcal D (\bar \phi)]_0 =&c \phi_0, \quad 
\mathcal B [\mathcal D (\bar \phi)]_{n+1} =c \phi_{n+1}, \\
    \mathcal B [\mathcal D (\bar \phi)]_i=&(c+2)\mathcal D(\bar \phi)_i=\frac{c+2}{2}\bar\phi_i,   \quad \text{if $i$ is odd, i.e., $x_i$ is a cell center};\notag\\
   \mathcal B [\mathcal D (\bar \phi)]_i= &\left(c+\frac72\right)D(\bar \phi)_i -2\left(1-\frac{h u_i}{2\mu}\right)D(\bar \phi)_{i-1}-2\left(1+\frac{h u_i}{2\mu}\right)D(\bar \phi)_{i+1}\\
 =&\left(c+\frac72\right) \phi_i -\left(1-\frac{h u_i}{2\mu}\right) \phi_{i-1}-\left(1+\frac{h u_i}{2\mu}\right)\phi_{i+1}\\
 & \text{if $i$ is even, i.e., $x_i$ is a cell end}.\notag\\
  \end{align*}
Then it is easy to see that $\mathcal B [\mathcal D (\bar {\mathbf 1})]>0$ thus $BD\mathbf 1>0$ for any $c>0$. 
By Theorem \ref{rowsumcondition-thm2}, $BD$ has positive row sums thus $A_d+A^z=B$ is a nonsingular M-matrix. 

\subsubsection{Verification of $A^+_a\leq A^z A_d^{-1} A^s$}

Since $A^z A_d^{-1} A^s\geq 0$, in order to verify $A^+_a\leq A^z A_d^{-1} A^s$, we only need to compare $\mathcal A^+_a(\bar \phi)_i$ with   $\mathcal A^z [\mathcal A^{-1}_d (\mathcal A^s (\bar \phi))]_i$ for even $i$.
For a cell end $x_i$, we have
\begin{align*}
\mathcal A^z [\mathcal A^{-1}_d (\mathcal A^s (\bar \phi))]_i =&-2\left(1-\frac{h u_i}{2\mu}\right)\ \mathcal A^{-1}_d (\mathcal A^s (\bar \phi))_{i-1}  -2\left(1+\frac{h u_i}{2\mu}\right) \mathcal A^{-1}_d (\mathcal A^s (\bar \phi))_{i+1} \\
=&-2\left(1-\frac{h u_i}{2\mu}\right)\frac{1}{c+2}\mathcal A^s (\bar \phi)_{i-1}  -2\left(1+\frac{h u_i}{2\mu}\right)\frac{1}{c+2}\mathcal A^s (\bar \phi)_{i+1} \\
=&2\frac{\left(1-\frac{h u_i}{2\mu}\right)\left(1-\frac{h u_{i-1}}{2\mu}\right)}{c+2}\phi_{i-2}
+2\frac{\left(1-\frac{h u_i}{2\mu}\right)\left(1+\frac{h u_{i-1}}{2\mu}\right)}{c+2}\phi_{i}\\
&+2\frac{\left(1+\frac{h u_i}{2\mu}\right)\left(1-\frac{h u_{i+1}}{2\mu}\right)}{c+2}\phi_{i}+2\frac{\left(1+\frac{h u_i}{2\mu}\right)\left(1+\frac{h u_{i+1}}{2\mu}\right)}{c+2}\phi_{i+2}.
\end{align*}
It suffices to have 
\[\left(1+\frac{hu_i}{\mu}\right)\frac14 \leq 2\frac{\left(1-\frac{h u_i}{2\mu}\right)\left(1-\frac{h u_{i-1}}{2\mu}\right)}{c+2},\left(1-\frac{hu_i}{\mu}\right)\frac14 \leq 2\frac{\left(1+\frac{h u_i}{2\mu}\right)\left(1+\frac{h u_{i+1}}{2\mu}\right)}{c+2},\]
which are equivalent to 
\[\resizebox{\textwidth}{!} 
{$
(12+2c)\frac{h u_i}{2\mu}+8\frac{h u_{i-1}}{2\mu} - 8\frac{h u_{i-1}}{2\mu} \frac{h u_i}{2\mu}\leq 6-c,-(12+2c)\frac{h u_i}{2\mu}-8\frac{h u_{i+1}}{2\mu} - 8\frac{h u_{i-1}}{2\mu} \frac{h u_i}{2\mu}\leq 6-c.
$}\]
Let $a=\max_i|u_i|\frac{h}{2\mu}$, then it suffices to require
\[(12+2c)a+8a+8a^2\leq 6-c \Longleftrightarrow 8a^2+(20+2c)a-(6-c)\leq 0.\]
From the inequality above, we get $a\leq \frac{\sqrt{(c+6)^2+112}-(c+10)}{8}$ for a fixed $c>0$. 
Since $\frac{\sqrt{(c+6)^2+112}-(c+10)}{8}>0$ implies $c<6$, we have $c\in(0,6)$.

For a fixed $a>0$, then we need $c\leq \frac{-8a^2 -20a +6}{2a+1}$. For $\frac{-8a^2 -20a +6}{2a+1}>0$, we must have $a<\frac{\sqrt{37}-5}{4}$
\subsubsection{Sufficient conditions in 1-D}
Now we can summarize all the constraints to apply Theorem \ref{newthm3}.
\begin{theorem}
\label{thm-monotonicity2}
Let $\|u\|_\infty=\max_i|u_i|$.
For the scheme \eqref{explicitscheme-1d} to be inverse positive, i.e., $\bar L^{-1}\geq 0$, the following  conditions are sufficient:
\begin{itemize}
\item For a mesh size $h$ satisfying $h\frac{\|u\|_\infty}{2\mu}=a<\frac{\sqrt{37}-5}{4}\approx 0.271$, time step $\Delta t$ satisfies $\Delta t\frac{\mu}{h^2}\geq\frac{2a+1} {-8a^2 -20a +6}$.
\item For a time step $\Delta t$ satisfying $\Delta t\frac{\mu}{h^2}=\frac{1}{c}>\frac16$, the mesh size $h$ satisfies  $h\frac{\|u\|_\infty}{\mu}\leq \frac{\sqrt{(c+6)^2+112}-(c+10)}{4}$.
 \end{itemize}
In particular, the following are convenient explicit sufficient mesh constraints for the inverse positivity:
\begin{itemize}
\item For a mesh size $h$ satisfying $h\frac{\|u\|_\infty}{\mu}\leq \frac12$, time step satisfies $\Delta t\frac{\mu}{h^2}\geq3$.
\item For a time step  $\Delta t$ satisfying $\Delta t\frac{\mu}{h^2}\geq \frac12 $, the mesh size $h$ satisfies  $h\frac{\|u\|_\infty}{\mu}\leq\frac12$.
 \end{itemize}

\end{theorem}

\subsection{Verification of Lorenz's condition: two-dimensional case}

For simplicity, we only consider the case $\Delta x=\Delta y=h$. 
Let $c=\frac{h^2}{\mu\Delta t}$.
For the linear operator $\mathcal A=\frac{h^2}{\mu\Delta t}\bar{\mathcal  L}:\mathbbm R^{(n_y+2)\times(n_x+2)}\rightarrow R^{(n_y+2)\times(n_x+2)}$, we have
\begin{align*}
{\mathcal A}(\bar \phi)_{ij}&=c\phi_{ij},\quad \text{$(x_i, y_j)$ is a boundary point;}\\
{\mathcal A}(\bar \phi)_{ij}&=c\phi_{ij}+\frac{h u_{ij}}{2\mu}(\phi_{i+1,j}-\phi_{i-1,j})+\frac{h v_{ij}}{2\mu}(\phi_{i,j+1}-\phi_{i,j-1})\\
 &+(-\phi_{i-1,j}+2\phi_{ij}-\phi_{i+1,j})+(-\phi_{i,j-1}+2\phi_{ij}-\phi_{i,j+1}), \text{ if $(x_i, y_j)$ is a cell center;}\\
{\mathcal A}(\bar \phi)_{ij}&=c\phi_{ij}+ \frac{h u_{ij}}{2\mu} (\phi_{i+1,j}-\phi_{i-1,j})+  \frac{h v_{ij}}{4\mu}(\phi_{i,j-2}-4\phi_{i,j-1}+4\phi_{i,j+1}-\phi_{i,j+2})\\
&+(-\phi_{i-1,j}+2\phi_{i,j}-\phi_{i+1,j})+\frac{\phi_{i,j-2}-8\phi_{i,j-1}+14\phi_{i,j}-8\phi_{i,j+1}+\phi_{i,j+2}}{4}\\
&\text{if $(x_i, y_j)$ is an interior edge center for an edge parallel to x-axis};\\
{\mathcal A}(\bar \phi)_{ij}&=c\phi_{ij}+  \frac{h u_{ij}}{4\mu}(\phi_{i-2,j}-4\phi_{i-1,j}+4\phi_{i+1,j}-\phi_{i+2,j})+\frac{h v_{ij}}{2\mu}(\phi_{i,j+1}-\phi_{i,j-1})\\
& +\frac{\phi_{i-2,j}-8\phi_{i-1,j}+14\phi_{i,j}-8\phi_{i+1,j}+\phi_{i+2,j}}{4}+(-\phi_{i,j-1}+2\phi_{i,j}-\phi_{i,j+1}),\\
&\text{if $(x_i, y_j)$ is an interior edge center for an edge parallel to y-axis};\\
{\mathcal A}(\bar \phi)_{ij}&=c\phi_{ij}+   \frac{h u_{ij}}{\mu}\frac{\phi_{i-2,j}-4\phi_{i-1,j}+4\phi_{i+1,j}-\phi_{i+2,j}}{4}+ \frac{h v_{ij}}{\mu}\frac{\phi_{i,j-2}-4\phi_{i,j-1}+4\phi_{i,j+1}-\phi_{i,j+2}}{4}\\\
&+\frac{\phi_{i-2,j}-8\phi_{i-1,j}+14\phi_{i,j}-8\phi_{i+1,j}+\phi_{i+2,j}}{4}+\frac{\phi_{i,j-2}-8\phi_{i,j-1}+14\phi_{i,j}-8\phi_{i,j+1}+\phi_{i,j+2}}{4}\\
&\text{if $(x_i, y_j)$ is an interior knot}.\\
\end{align*}

  \subsubsection{Splitting of negative off-diagonal entries}
In order to have fixed signs for all entries, we assume $h\max_{ij}\{|u_{ij}|,|v_{ij}|\}\leq 2\mu$ for all $i,j$.
Then we have
\begin{align*}
{\mathcal A}_d(\bar \phi)_{ij}&=c\phi_{ij},\quad \text{$(x_i, y_j)$ is a boundary point;}\\
{\mathcal A}_d(\bar \phi)_{ij}&=\left(c+4\right)\phi_{ij}, \text{$(x_i, y_j)$ is a cell center;}\\
{\mathcal A}_d(\bar \phi)_{ij}&=\left(c+\frac{11}{2}\right)\phi_{ij}, \text{$(x_i, y_j)$ is an interior edge center};\\
{\mathcal A}_d(\bar \phi)_{ij}&=\left(c+7\right)\phi_{ij}, \text{$(x_i, y_j)$ is an interior knot}.
\end{align*}
For positive off-diagonal parts, we have:
\begin{align*}
{\mathcal A}_a^+(\bar \phi)_{ij}&=0,\quad \text{$(x_i, y_j)$ is a boundary point;}\\
{\mathcal A}_a^+(\bar \phi)_{ij}&=\frac14\left(1+ \frac{h v_{ij}}{\mu}\right)\phi_{i,j-2}+\frac14\left(1- \frac{h v_{ij}}{\mu}\right)\phi_{i,j+2}, 
\\&\text{$(x_i, y_j)$ is an interior edge center for an edge parallel to x-axis};\\
{\mathcal A}_a^+(\bar \phi)_{ij}&=\frac14\left(1+ \frac{h u_{ij}}{\mu}\right)\phi_{i-2,j}+\frac14\left(1- \frac{h u_{ij}}{\mu}\right)\phi_{i+2,j}, 
\\&\text{$(x_i, y_j)$ is an interior edge center for an edge parallel to y-axis};\\
{\mathcal A}_a^+(\bar \phi)_{ij}&=\frac14\left(1+ \frac{h v_{ij}}{\mu}\right)\phi_{i,j-2}+\frac14\left(1- \frac{h v_{ij}}{\mu}\right)\phi_{i,j+2}+\frac14\left(1+ \frac{h u_{ij}}{\mu}\right)\phi_{i-2,j}+\frac14\left(1- \frac{h u_{ij}}{\mu}\right)\phi_{i+2,j},\\
&\text{ $(x_i, y_j)$ is an interior knot}.
\end{align*}
Then we defined a splitting $A^-_a=A^z+A^s$ as:
\begin{align*}
{\mathcal A}^z(\bar \phi)_{ij}&=0,\quad \text{$(x_i, y_j)$ is a boundary point;}\\
{\mathcal A}^z(\bar \phi)_{ij}&=0,\quad \text{$(x_i, y_j)$ is a cell center;}\\
{\mathcal A}^z(\bar \phi)_{ij}&=-2\left(1+ \frac{h   v_{ij}}{2\mu} \right)\phi_{i,j-1}-2\left(1- \frac{h   v_{ij}}{2\mu} \right)\phi_{i,j+1}\\
&\text{$(x_i, y_j)$ is an interior edge center for an edge parallel to x-axis};\\
{\mathcal A}^z(\bar \phi)_{ij}&=-2\left(1+ \frac{h   u_{ij}}{2\mu} \right)\phi_{i-1,j}-2\left(1- \frac{h   u_{ij}}{2\mu} \right)\phi_{i+1,j},\\
&\text{$(x_i, y_j)$ is an interior edge center for an edge parallel to y-axis};\\
{\mathcal A}^z(\bar \phi)_{ij}&=-2\left(1+ \frac{h   v_{ij}}{2\mu} \right)\phi_{i,j-1}-2\left(1- \frac{h   v_{ij}}{2\mu} \right)\phi_{i,j+1}-2\left(1+ \frac{h   u_{ij}}{2\mu} \right)\phi_{i-1,j}-2\left(1- \frac{h   u_{ij}}{2\mu} \right)\phi_{i+1,j},\\
&\text{$(x_i, y_j)$ is an interior knot}.\\
\end{align*}
\begin{align*}
{\mathcal A}^s(\bar \phi)_{ij}&=0,\quad \text{$(x_i, y_j)$ is a boundary point;}\\
{\mathcal A}^s(\bar \phi)_{ij}&=-\left(1-\frac{h u_{ij}}{2\mu}\right)\phi_{i+1,j}-\left(1+\frac{h u_{ij}}{2\mu}\right)\phi_{i-1,j}
-\left(1-\frac{h v_{ij}}{2\mu}\right)\phi_{i,j+1}-\left(1+\frac{h v_{ij}}{2\mu}\right)\phi_{i,j-1},\\& \text{$(x_i, y_j)$ is a cell center;}\\
{\mathcal A}^s(\bar \phi)_{ij}&=-\left(1-\frac{h u_{ij}}{2\mu}\right)\phi_{i+1,j}-\left(1+\frac{h u_{ij}}{2\mu}\right)\phi_{i-1,j}
\\
&\text{$(x_i, y_j)$ is an interior edge center for an edge parallel to x-axis};\\
{\mathcal A}^s(\bar \phi)_{ij}&=-\left(1-\frac{h v_{ij}}{2\mu}\right)\phi_{i,j+1}-\left(1+\frac{h v_{ij}}{2\mu}\right)\phi_{i,j-1},\\ 
&\text{$(x_i, y_j)$ is an interior edge center for an edge parallel to y-axis};\\
{\mathcal A}^s(\bar \phi)_{ij}&=0,\quad\text{ $(x_i, y_j)$ is an interior knot}.\\
\end{align*}
  \subsubsection{Verification of $A_d+A^z$ being an M-matrix}
  Let $\mathcal B=\mathcal A_d+\mathcal A^z$, then we have
  \begin{align*}
{\mathcal B}(\bar \phi)_{ij}&=c\phi_{ij},\quad \text{$(x_i, y_j)$ is a boundary point;}\\
{\mathcal B}(\bar \phi)_{ij}&=(c+4)\phi_{ij},\quad \text{$(x_i, y_j)$ is a cell center;}\\
{\mathcal B}(\bar \phi)_{ij}&=\left(c+\frac{11}{2}\right)\phi_{ij}-2\left(1+ \frac{h   v_{ij}}{2\mu} \right)\phi_{i,j-1}-2\left(1- \frac{h   v_{ij}}{2\mu} \right)\phi_{i,j+1},\\
&\text{$(x_i, y_j)$ is an interior edge center for an edge parallel to x-axis};\\
{\mathcal B}(\bar \phi)_{ij}&=\left(c+\frac{11}{2}\right)\phi_{ij}-2\left(1+ \frac{h   u_{ij}}{2\mu} \right)\phi_{i-1,j}-2\left(1- \frac{h   u_{ij}}{2\mu} \right)\phi_{i+1,j},\\
&\text{$(x_i, y_j)$ is an interior edge center for an edge parallel to y-axis};\\
{\mathcal B}(\bar \phi)_{ij}&=(c+7)\phi_{ij}-2\left(1+ \frac{h   v_{ij}}{2\mu} \right)\phi_{i,j-1}-2\left(1- \frac{h   v_{ij}}{2\mu} \right)\phi_{i,j+1}\\
&-2\left(1+ \frac{h   u_{ij}}{2\mu} \right)\phi_{i-1,j}-2\left(1- \frac{h   u_{ij}}{2\mu} \right)\phi_{i+1,j}, \\
&\text{$(x_i, y_j)$ is an interior knot}.
\end{align*}
Define a positive diagonal operator $\mathcal D$ as
  \begin{align*}
{\mathcal D}(\bar \phi)_{ij}&=\phi_{ij},\quad \text{$(x_i, y_j)$ is a boundary point, or a cell center, or an interior knot;}\\
{\mathcal D}(\bar \phi)_{ij}&=\frac34 \phi_{ij}, \quad\text{$(x_i, y_j)$ is an interior edge center.}\end{align*}
Then we have 
  \begin{align*}
\mathbf B[{\mathcal D}(\bar \phi)]_{ij}=&c\mathcal D(\bar \phi)_{ij}=c\phi_{ij},\quad \text{$(x_i, y_j)$ is a boundary point;}\\
\mathbf B[{\mathcal D}(\bar \phi)]_{ij}=&(c+4)\mathcal D(\bar \phi)_{ij}=(c+4)\phi_{ij},\quad \text{$(x_i, y_j)$ is a cell center;}\\
\mathbf B[{\mathcal D}(\bar \phi)]_{ij}=&\left(c+\frac{11}{2}\right)\mathcal D(\bar \phi)_{ij}-2\left(1+ \frac{h   v_{ij}}{2\mu} \right)\mathcal D(\bar \phi)_{i,j-1}-2\left(1- \frac{h   v_{ij}}{2\mu} \right)\mathcal D(\bar \phi)_{i,j+1}\\
=&\left(\frac34c+\frac{33}{8}\right)\phi_{ij}-2\left(1+ \frac{h   v_{ij}}{2\mu} \right)\phi_{i,j-1}-2\left(1- \frac{h   v_{ij}}{2\mu} \right)\phi_{i,j+1},\\
&\text{$(x_i, y_j)$ is an interior edge center for an edge parallel to x-axis};\\
\mathbf B[{\mathcal D}(\bar \phi)]_{ij}=&\left(c+\frac{11}{2}\right)\mathcal D(\bar \phi)_{ij}-2\left(1+ \frac{h   u_{ij}}{2\mu} \right)\mathcal D(\bar \phi)_{i-1,j}-2\left(1- \frac{h   u_{ij}}{2\mu} \right)\mathcal D(\bar \phi)_{i+1,j}\\
=&\left(\frac34c+\frac{33}{8}\right)\phi_{ij}-2\left(1+ \frac{h   u_{ij}}{2\mu} \right)\phi_{i-1,j}-2\left(1- \frac{h   u_{ij}}{2\mu} \right)\phi_{i+1,j},\\
&\text{$(x_i, y_j)$ is an interior edge center for an edge parallel to y-axis};\\
\mathbf B[{\mathcal D}(\bar \phi)]_{ij}=&(c+7)\mathcal D(\bar \phi)_{ij}-2\left(1+ \frac{h   v_{ij}}{2\mu} \right)\mathcal D(\bar \phi)_{i,j-1}-2\left(1- \frac{h   v_{ij}}{2\mu} \right)\mathcal D(\bar \phi)_{i,j+1}\\
&-2\left(1+ \frac{h   u_{ij}}{2\mu} \right)\mathcal D(\bar \phi)_{i-1,j}-2\left(1- \frac{h   u_{ij}}{2\mu} \right)\mathcal D(\bar \phi)_{i+1,j},\\
=&(c+7)\phi_{ij}-\frac32\left(1+ \frac{h   v_{ij}}{2\mu} \right)\phi_{i,j-1}-\frac32\left(1- \frac{h   v_{ij}}{2\mu} \right)\phi_{i,j+1}\\
&-\frac32\left(1+ \frac{h   u_{ij}}{2\mu} \right)\phi_{i-1,j}-\frac32\left(1- \frac{h   u_{ij}}{2\mu} \right)\phi_{i+1,j},\\
&\text{$(x_i, y_j)$ is an interior knot}.
\end{align*}
And the matrix $BD$ has positive row sums for any $c>0$ due to the following:
  \begin{align*}
\mathbf B[{\mathcal D}(\bar{\mathbf 1})]_{ij}=&c,\quad \text{$(x_i, y_j)$ is a boundary point;}\\
\mathbf B[{\mathcal D}(\bar{\mathbf 1})]_{ij}=&c+4,\quad \text{$(x_i, y_j)$ is a cell center;}\\
\mathbf B[{\mathcal D}(\bar{\mathbf 1})]_{ij}=&\frac34c+\frac18, \quad \text{$(x_i, y_j)$ is an interior edge center};\\
\mathbf B[{\mathcal D}(\bar{\mathbf 1})]_{ij}=&c+1,\quad \text{$(x_i, y_j)$ is an interior knot}.
\end{align*} 

 By Theorem \ref{rowsumcondition-thm2}, $A_d+A^z=B$ is a nonsingular M-matrix. 
 
  \subsubsection{Verification of $A^+_a\leq A^z A_d^{-1} A^s$ for edge centers}
Next we verify  $A^+_a\leq A^z A_d^{-1} A^s$. We first compare   $\mathcal A^z [\mathcal A^{-1}_d (\mathcal A^s (\bar \phi))]_{ij}$ with $\mathcal A^+_a(\bar \phi)_{ij}$ for the case that $(x_i, y_j)$ is an interior edge center for an edge parallel to x-axis.
\begin{align*}
&\mathcal A^z [\mathcal A^{-1}_d (\mathcal A^s (\bar \phi))]_{ij}
\\&=-2\left(1+ \frac{h   v_{ij}}{2\mu} \right)\mathcal A^{-1}_d (\mathcal A^s (\bar \phi))_{i,j-1}-2\left(1- \frac{h   v_{ij}}{2\mu} \right)\mathcal A^{-1}_d (\mathcal A^s (\bar \phi))_{i,j+1}\\
&=-2\left(1+ \frac{h   v_{ij}}{2\mu} \right)\left(c+4\right)^{-1} \mathcal A^s (\bar \phi)_{i,j-1}-2\left(1- \frac{h   v_{ij}}{2\mu} \right)\left(c+4\right)^{-1} \mathcal A^s (\bar \phi)_{i,j+1}\\
&\resizebox{\hsize}{!}{$=2\frac{1+ \frac{h   v_{ij}}{2\mu}}{c+4}\left[\left(1-\frac{h u_{i,j-1}}{2\mu}\right)\phi_{i+1,j-1}+\left(1+\frac{h u_{i,j-1}}{2\mu}\right)\phi_{i-1,j-1}
+\left(1-\frac{h v_{i,j-1}}{2\mu}\right)\phi_{i,j}+\left(1+\frac{h v_{i,j-1}}{2\mu}\right)\phi_{i,j-2}\right]$}\\
&\resizebox{\hsize}{!}{$+2\frac{1- \frac{h   v_{ij}}{2\mu}}{c+4}\left[\left(1-\frac{h u_{i,j+1}}{2\mu}\right)\phi_{i+1,j+1}+\left(1+\frac{h u_{i,j+1}}{2\mu}\right)\phi_{i-1,j+1}
+\left(1-\frac{h v_{i,j+1}}{2\mu}\right)\phi_{i,j+2}+\left(1+\frac{h v_{i,j+1}}{2\mu}\right)\phi_{i,j}\right]. $}
\end{align*}
For  $A^+_a\leq A^z A_d^{-1} A^s$ to hold, we need
\[2\frac{1+ \frac{h   v_{ij}}{2\mu}}{c+4}\left(1+\frac{h v_{i,j-1}}{2\mu}\right)\geq \frac14\left(1+ \frac{h v_{ij}}{\mu}\right), \quad 2\frac{1- \frac{h   v_{ij}}{2\mu}}{c+4}\left(1-\frac{h v_{i,j+1}}{2\mu}\right)\geq \frac14\left(1- \frac{h v_{ij}}{\mu}\right), \]
which are equivalent to 
\begin{align*}
&2c\frac{h   v_{ij}}{2\mu}-8\frac{h v_{i,j-1}}{2\mu}-8\frac{h   v_{ij}}{2\mu}\frac{h v_{i,j-1}}{2\mu}\leq 4-c, \\
&-2c\frac{h   v_{ij}}{2\mu}+8\frac{h v_{i,j+1}}{2\mu}-8\frac{h   v_{ij}}{2\mu}\frac{h v_{i,j+1}}{2\mu}\leq 4-c.
\end{align*}
Let $a=\max_{i,j}|v_{ij}|\frac{h}{2\mu}>0$, then it suffices to have
\[2ca+8a+8a^2\leq 4-c \Longleftrightarrow 8a^2+(8+2c)a-(4-c)\geq 0 \Longleftrightarrow 0<  a\leq  \frac{\sqrt{c^2+48}-(c+4)}{8}.\]
For $\frac{\sqrt{c^2+48}-(c+4)}{8}>0$ to hold, we need $c<4$.

For fixed $a>0$, we have 
$$2ca+8a+8a^2\leq 4-c \Longleftrightarrow c\leq \frac{4-8a^2-8a}{1+2a}=\frac{-8(a+\frac12)^2+6}{2a+1}.$$ For $\frac{-8(a+\frac12)^2+6}{2a+1}$ to be positive, we need $a<\frac{\sqrt3-1}{2}$

 For the case that $(x_i, y_j)$ is an interior edge center for an edge parallel to y-axis, the discussion will be similar and the same mesh constraints apply due to the symmetry.

We summarize the constraints obtained so far as the following:
\begin{itemize}
\item For any $c\in(0,4)$, $0<  a\leq  \frac{\sqrt{c^2+48}-(c+4)}{8}$;
\item For any $a\in(0, \frac{\sqrt3-1}{2})$, $0<c\leq c\leq \frac{4-8a^2-8a}{1+2a}.$
\end{itemize}

  \subsubsection{Verification of $A^+_a\leq A^z A_d^{-1} A^s$ for knots}
Next, consider the case $(x_i, y_j)$ is an interior knot. 
\begin{align*}
&\mathcal A^z [\mathcal A^{-1}_d (\mathcal A^s (\bar \phi))]_{ij}\\
=&-2\left(1+ \frac{h   v_{ij}}{2\mu} \right)\mathcal A^{-1}_d (\mathcal A^s (\bar \phi))_{i,j-1}-2\left(1- \frac{h   v_{ij}}{2\mu} \right)\mathcal A^{-1}_d (\mathcal A^s (\bar \phi))_{i,j+1}\\
&-2\left(1+ \frac{h   u_{ij}}{2\mu} \right)\mathcal A^{-1}_d (\mathcal A^s (\bar \phi))_{i-1,j}-2\left(1- \frac{h   u_{ij}}{2\mu} \right)\mathcal A^{-1}_d (\mathcal A^s (\bar \phi))_{i+1,j},\\
=&-2\frac{1+ \frac{h   v_{ij}}{2\mu}}{c+\frac{11}{2}} \mathcal A^s (\bar \phi)_{i,j-1}-2\frac{1- \frac{h   v_{ij}}{2\mu}}{c+\frac{11}{2}}  \mathcal A^s (\bar \phi)_{i,j+1}-2\frac{1+ \frac{h   u_{ij}}{2\mu}}{c+\frac{11}{2}} \mathcal A^s (\bar \phi)_{i-1,j}-2\frac{1- \frac{h   u_{ij}}{2\mu}}{c+\frac{11}{2}} \mathcal A^s (\bar \phi)_{i+1,j},\\
=&2\frac{1+ \frac{h   v_{ij}}{2\mu}}{c+\frac{11}{2}}\left[\left(1-\frac{h v_{i,j-1}}{2\mu}\right)\phi_{i,j}+\left(1+\frac{h v_{i,j-1}}{2\mu}\right)\phi_{i,j-2}\right] \\
&+2\frac{1- \frac{h   v_{ij}}{2\mu}}{c+\frac{11}{2}}\left[\left(1-\frac{h v_{i,j+1}}{2\mu}\right)\phi_{i,j+2}+\left(1+\frac{h v_{i,j+1}}{2\mu}\right)\phi_{i,j}\right] \\
&+2\frac{1+ \frac{h   u_{ij}}{2\mu}}{c+\frac{11}{2}} \left[\left(1-\frac{h u_{i-1,j}}{2\mu}\right)\phi_{i,j}+\left(1+\frac{h u_{i-1,j}}{2\mu}\right)\phi_{i-2,j}\right]\\
&+2\frac{1- \frac{h   u_{ij}}{2\mu}}{c+\frac{11}{2}} \left[\left(1-\frac{h u_{i+1,j}}{2\mu}\right)\phi_{i+2,j}+\left(1+\frac{h u_{i+1,j}}{2\mu}\right)\phi_{i,j}\right].
\end{align*}
For  $A^+_a\leq A^z A_d^{-1} A^s$ to hold, we need
\begin{align*}
&2\frac{1+ \frac{h   v_{ij}}{2\mu}}{c+\frac{11}{2}} \left(1+\frac{h v_{i,j-1}}{2\mu}\right) \geq \frac14\left(1+ \frac{h v_{ij}}{\mu}\right),
\quad 2\frac{1- \frac{h   v_{ij}}{2\mu}}{c+\frac{11}{2}}\left(1-\frac{h v_{i,j+1}}{2\mu}\right)\geq \frac14\left(1- \frac{h v_{ij}}{\mu}\right), \\
& 2\frac{1+ \frac{h   u_{ij}}{2\mu}}{c+\frac{11}{2}} \left(1+\frac{h u_{i-1,j}}{2\mu}\right)\geq \frac14\left(1+ \frac{h u_{ij}}{\mu}\right),
\quad 2\frac{1- \frac{h   u_{ij}}{2\mu}}{c+\frac{11}{2}} \left(1-\frac{h u_{i+1,j}}{2\mu}\right)\geq \frac14\left(1- \frac{h u_{ij}}{\mu}\right),
\end{align*}
which are equivalent to 
\begin{align*}
 (3+2c) \frac{h v_{ij}}{2\mu}-8 \frac{h v_{i,j-1}}{2\mu} -8\frac{h v_{ij}}{2\mu}\frac{h v_{i,j-1}}{2\mu} &\leq \frac52-c,\\
 -(3+2c) \frac{h v_{ij}}{2\mu}+8 \frac{h v_{i,j-1}}{2\mu} -8\frac{h v_{ij}}{2\mu}\frac{h v_{i,j-1}}{2\mu}& \leq \frac52-c,\\ 
 (3+2c) \frac{h u_{ij}}{2\mu}-8 \frac{h u_{i-1,j}}{2\mu} -8\frac{h u_{ij}}{2\mu}\frac{h u_{i-1,j}}{2\mu} &\leq \frac52-c,\\
 -(3+2c) \frac{h u_{ij}}{2\mu}+8 \frac{h u_{i-1,j}}{2\mu} -8\frac{h u_{ij}}{2\mu}\frac{h u_{i-1,j}}{2\mu}& \leq \frac52-c.
\end{align*}
Let $a=\max\{\max_{i,j}|v_{ij}|, \max_{i,j}|u_{ij}|\}\frac{h}{2\mu}>0$, then it suffices to require
\[\resizebox{\hsize}{!}{$
(3+2c)a+8a+8a^2\leq \frac52-c\Longleftrightarrow 16a^2+(11+2c)2a-5+2c\leq 0 \Longleftarrow  0\leq a\leq \frac{\sqrt{(c+\frac32)^2+48}-(c+\frac{11}{2})}{8}. $}\]
To ensure $\frac{\sqrt{(c+\frac32)^2+48}-(c+\frac{11}{2})}{8}>0,$ we need $c<\frac52$, with which we  have $\frac{\sqrt{(c+\frac32)^2+48}-(c+\frac{11}{2})}{8}< \frac{\sqrt{c^2+48}-(c+4)}{8}$.
For a fixed $a>0$, we also  have 
$$(3+2c)a+8a+8a^2\leq \frac52-c\Longleftrightarrow 0< c\leq \frac{-8a^2-11a+\frac52}{2a+1}.$$ 
For $\frac{-8a^2-11a+\frac52}{2a+1}>0$, we need $a<\frac{\sqrt{201}-11}{16}$ which is smaller than $\frac{\sqrt{3}-1}{2}$.
When $\frac{-8a^2-11a+\frac52}{2a+1}>0$, we also have $\frac{-8a^2-11a+\frac52}{2a+1}< \frac{4-8a^2-8a}{1+2a}$.

Now we can summarize all mesh constraints for $A^+_a\leq A^z A_d^{-1} A^s$  at both edge centers and knots:
\begin{itemize}
\item For any $c\in(0,\frac52)$, $0<  a\leq  \frac{\sqrt{(c+\frac32)^2+48}-(c+\frac{11}{2})}{8}$;
\item For any $a\in(0, \frac{\sqrt{201}-11}{16})$, $0<c \leq\frac{-8a^2-11a+\frac52}{2a+1}.$
\end{itemize}
\subsubsection{Sufficient conditions in 2-D}
Since $ A^z A_d^{-1} A^s\geq 0$ and $\mathcal A^+_a(\phi)_{ij}$ are nonzero only at interior knots and interior edge centers, we have already found all constraints to ensure $A^+_a\leq A^z A_d^{-1} A^s$. By applying Theorem \ref{newthm3}, we get the monotonicity result for the high order scheme:
\begin{theorem}
\label{thm-monotonicity3}
Let $\|\mathbf u\|_\infty=\max_{ij}\{|u_{ij},|u_{ij}|\}|$.
For the fourth order accurate scheme  \eqref{fullscheme} to be inverse positive, i.e., $\bar L^{-1}\geq 0$, the following  conditions are sufficient:
\begin{itemize}
\item For a mesh size $h$ satisfying $h\frac{\|\mathbf u\|_\infty}{2\mu}=a<\frac{\sqrt{201}-11}{16}\approx 0.199$, time step $\Delta t$ satisfies $\Delta t\frac{\mu}{h^2}\geq\frac{2a+1} {-8a^2 -11a +\frac52}$.
\item For a time step $\Delta t$ satisfying $\Delta t\frac{\mu}{h^2}=\frac{1}{c}>\frac25$, the mesh size $h$ satisfies  $h\frac{\|\mathbf u\|_\infty}{\mu}\leq \frac{\sqrt{(c+\frac32)^2+48}-(c+\frac{11}{2})}{8}$.
 \end{itemize}
In particular, the following are convenient explicit sufficient mesh constraints for the inverse positivity:
\begin{itemize}
\item For a mesh size $h$ satisfying $h\frac{\|\mathbf u\|_\infty}{\mu}\leq \frac13$, time step satisfies $\Delta t\frac{\mu}{h^2}\geq3$.
\item For a time step  $\Delta t$ satisfying $\Delta t\frac{\mu}{h^2}\geq 1 $, the mesh size $h$ satisfies  $h\frac{\|\mathbf u\|_\infty}{\mu}\leq \frac{\sqrt{217}-13}{8}\approx 0.216.$
 \end{itemize}
\end{theorem}

\section{The generalized Allen-Cahn equation}
\label{sec-ac}
We now consider \eqref{generalized-Allen-Cahn}.
We shall assume that the free energy
functional $F(\phi)$  has a double well form with minima at $\pm\beta$, where $\beta$ satisfies
satisfies
\begin{subequations}
\label{assumption}
\begin{equation}
\label{assumption1}
F'(\beta)=F'(-\beta)=0,\end{equation}
$F'(\beta)$ satisfies the monotone conditions away from $(-\beta, \beta)$:
\begin{equation}
\label{assumption2}
F'(\phi)<0, \forall \phi<-\beta;\quad F'(\beta)>0, \forall \phi>\beta.\end{equation}
\end{subequations}
Such energy functionals include the polynomial energy $F(\phi)=\frac14(\phi^2-1)^2$ with $\beta=1$,
and the logarithmic energy \eqref{log-energy}
with $\beta$ given by
$\frac{1}{2\beta}\ln \frac{1+\beta}{1-\beta}=\frac{\theta_c}{\theta}$, see \cite{shen2016maximum}.

By Lemma 2.1 in \cite{shen2016maximum}, we have
\begin{lemma}
\label{thm-lemma}
For an energy function satisfying \eqref{assumption} and $f(x)=x-\frac{\Delta t}{\varepsilon}F'(x)$, the following bounds of $f(x)$ hold:
\[  f(x)\in[-\beta, \beta], \quad \forall x\in[-\beta,\beta],\]
under the time step  constraint
$$\Delta t \max_{x\in[-\beta,\beta]} F''(x)\leq \varepsilon.$$
\end{lemma}

Consider a first order implicit explicit time (IMEX) discretization with the fourth order difference scheme:
\begin{equation}
\resizebox{\textwidth}{!} 
{$\begin{split}
 \phi^{n+1}+\Delta t \left[u^{n+1}.*(R_y \bar \phi^{n+1} D_{1x}^T)+v^{n+1}.*(D_{1y} \bar \phi^{n+1} R_x^T)-\mu(R_y \bar \phi^{n+1} D_{2x}^T+ D_{2y} \bar \phi^{n+1} R_x^T)\right]\\
 =\phi^{n}-\frac{\Delta t}{\varepsilon}F'(\phi^n),
 \end{split}$}
 \label{highorderscheme-allencahn}
 \end{equation}
where $F'(\phi^n)$ is a matrix with entries $F'(\phi^n_{ij})$. 

Notice that the time step has a lower bound $\Delta t>\frac25\frac{h^2}{\mu}$ in Theorem \ref{thm-monotonicity3} and an upper bound $\Delta t\leq \frac{\varepsilon}{\max\limits_{x\in[-\beta,\beta]} F''(x)}$ in Lemma \ref{thm-lemma}, thus we need an upper bound on the mesh size $h<\sqrt{\frac52\frac{\mu\varepsilon}{\max\limits_{x\in[-\beta,\beta]} F''(x)  }}$ so that the time step interval is not empty. Combined with Theorem \ref{thm-monotonicity3}, we have:
\begin{theorem}
\label{theorem-allen-cahn-imex}
The scheme \eqref{highorderscheme-allencahn} for homogeneous Dirichlet boundary condition satisfies the discrete maximum principle:
$$\min_{i,j}\phi^{n}\leq \phi^{n+1}_{ij}\leq \max_{i,j}\phi^{n},$$
under the following convenient mesh and time step constraints:
\begin{enumerate}
\item $$h\leq \min\left\{\frac13 \frac{\mu}{\|\mathbf u\|_\infty}, \sqrt{\frac{3\mu\varepsilon}{\max\limits_{x\in[-\beta,\beta]} F''(x)  }}\right\},\quad 3 \frac{h^2}{\mu}  \leq   \Delta t\leq \frac{\varepsilon}{\max\limits_{x\in[-\beta,\beta]} F''(x)}.$$
\item $$h\leq \min\left\{ 0.216\frac{\mu}{\|\mathbf u\|_\infty}, \sqrt{\frac{\mu\varepsilon}{\max\limits_{x\in[-\beta,\beta]} F''(x)  }}\right\},\quad  \frac{h^2}{\mu}  \leq   \Delta t\leq \frac{\varepsilon}{\max\limits_{x\in[-\beta,\beta]} F''(x)}.$$
 \end{enumerate}
\end{theorem}

\begin{rmk}
As a comparison, for the second order scheme with IMEX time discretization to satisfy the discrete maximum, there is no lower bound on the time step. 
By Theorem \ref{thm-monotonicity1}, we the mesh constraints need for second order scheme to be bound-preserving:
\[ h\leq \min 2 \frac{\mu}{\|\mathbf u\|_\infty},\quad \Delta t\leq \frac{\varepsilon}{\max\limits_{x\in[-\beta,\beta]} F''(x)}.\]
\end{rmk}

Next we consider the stabilized IMEX time discretization with the fourth order difference scheme with $S\geq 0$:
\begin{equation}
\resizebox{\textwidth}{!} 
{$\begin{split}
 \phi^{n+1}+\Delta t \left[S\phi^{n+1}+u^{n+1}.*(R_y \bar \phi^{n+1} D_{1x}^T)+v^{n+1}.*(D_{1y} \bar \phi^{n+1} R_x^T)-\mu(R_y \bar \phi^{n+1} D_{2x}^T+ D_{2y} \bar \phi^{n+1} R_x^T)\right]\\
 =\phi^{n}-\frac{\Delta t}{\varepsilon}F'(\phi^n)+S\Delta t\phi^n.
 \end{split}$}
 \label{highorderscheme-allencahn-stable}
 \end{equation}
Notice that  the stabilized IMEX time discretization \eqref{imex2} can be written as
\[ \frac{\phi^{n+1}-\phi^n}{\widetilde {\Delta t}}+u^{n+1}\phi^{n+1}_x+v^{n+1} \phi^{n+1}_y=\mu \Delta \phi^{n+1}-\frac{F'(\phi^n)}{\varepsilon} \]
with $\widetilde {\Delta t}=\frac{\Delta t}{1+\Delta tS}.$
Replacing $\Delta t$ by $\widetilde {\Delta t}$ in Theorem \ref{theorem-allen-cahn-imex}, we can easily get 
the result for the stabilized IMEX time discretization:
\begin{theorem}
The scheme \eqref{highorderscheme-allencahn-stable} with $S\geq 0$ for homogeneous Dirichlet boundary condition satisfies the discrete maximum principle:
$$\min_{i,j}\phi^{n}\leq \phi^{n+1}_{ij}\leq \max_{i,j}\phi^{n},$$
under the following convenient mesh and time step constraints:
\begin{enumerate}
\item $$h\leq \min\left\{\frac13 \frac{\mu}{\|\mathbf u\|_\infty}, \sqrt{\frac{3\mu\varepsilon}{\max\limits_{x\in[-\beta,\beta]} F''(x)  }}\right\}, 3 \frac{h^2}{\mu}  \leq   \frac{\Delta t}{1+\Delta tS}\leq \frac{\varepsilon}{\max\limits_{x\in[-\beta,\beta]} F''(x)};$$
\item $$h\leq \min\left\{ 0.216\frac{\mu}{\|\mathbf u\|_\infty}, \sqrt{\frac{\mu\varepsilon}{\max\limits_{x\in[-\beta,\beta]} F''(x)  }}\right\},  \frac{h^2}{\mu}  \leq   \frac{\Delta t}{1+\Delta tS}\leq \frac{\varepsilon}{\max\limits_{x\in[-\beta,\beta]} F''(x)}.$$
 \end{enumerate}
\end{theorem}

\begin{rmk} 
As a comparison, for the spatially second order scheme with the  stabilized IMEX time discretization to satisfy the discrete maximum, there is no lower bound on the time step, and 
by Theorem \ref{thm-monotonicity1},  the mesh constraints  for this scheme to be bound-preserving are:
\[ h\leq \min 2 \frac{\mu}{\|\mathbf u\|_\infty},\quad \frac{1}{\Delta t}+S\geq \frac{\max\limits_{x\in[-\beta,\beta]} F''(x)}{\varepsilon}.\]
\end{rmk}

\section{Stream function vorticity formulation of 2D incompressible flow}
\label{sec-ns}
The results in Section \ref{sec-mono} can also be used to construct
a bound-preserving scheme for any passive convection-diffusion with an incompressible velocity field. As an example,
we consider the two-dimensional incompressible Navier-Stokes equation in stream function vorticity form:
\begin{align*}
&\omega_t+u\omega_x+v\omega_y=\mu \Delta \omega \\
& \Delta \psi=\omega,\quad (u, v)=(-\psi_y, \psi_x)\\
&\omega(x,y,0)=\omega_0(x,y), \quad (x,y)\in \Omega,
\end{align*}
where $\omega$ is the vorticity, and $\psi$ is the stream function. For simplicity, we only consider homogeneous Dirichlet boundary conditions, and extensions to  periodic boundary conditions are straightforward. We consider a first order time discretization:
\begin{align*}
& \Delta \psi^{n+1}=\omega^{n},\quad (u^{n+1}, v^{n+1})=(-\psi^{n+1}_y, \psi^{n+1}_x),\\
&\frac{\omega^{n+1}-\omega^{n}}{\Delta t}+u^{n+1} \omega^{n+1}_x+v^{n+1}\omega^{n+1}_y=\mu \Delta \omega^{n+1}, 
\end{align*}
with second order or fourth order finite difference spatial discretization as described in Section \ref{sec-scheme}.

When using the fourth order scheme for $\omega$, the same fourth order scheme can also be used to solve the Poisson equation $ \Delta \psi=\omega$. See \cite{li2019fourth} for an efficient inversion of the discrete Laplacian via an eigenvector method.
Once $\psi$ is obtained from the Poisson equation, the velocity field can be computed by taking finite difference of $\psi$. However, for a fourth order scheme, the difference matrix $D_1$ cannot be used because it is only a second order finite difference approximating first order derivatives. Instead, a conventional fourth order finite difference operator should be used to compute the numerical differentiation.

Notice that the scheme for $\omega$ here is the same as \eqref{fullscheme}. 
With Theorem \ref{thm-dmp}, the fully discrete   scheme with homogeneous Dirichlet boundary condition satisfies the discrete maximum principle 
$$\min_{i,j}\omega^{n}\leq \omega^{n+1}_{ij}\leq \max_{i,j}\omega^{n},$$
if the mesh size and time step constraints in Theorem \ref{thm-monotonicity1}, Theorem \ref{thm-monotonicity2} and Theorem \ref{thm-monotonicity3} are satisfied.

\section{Numerical tests}
\label{sec-test}
For implementation of the scheme, biconjugate gradient stabilized (BiCGSTAB) method is used for solving the linear system with variable coefficients at each time, with the discrete Laplacian as a preconditioner which can be efficiently inverted,
see \cite{li2020superconvergence} for details. 

\subsection{Accuracy test}
We first test accuracy for the generalized Allen-Cahn equation with $F(\phi)=\frac14(\phi^2-1)^2$, parameters $\mu=0.1$ and $\varepsilon=0.05$ and a given velocity field
\[ u=v=\sin(y-x).\]

A source term is added so that the exact solution is $$\phi=(0.75+0.25\sin(t))\sin y \sin^2 x.$$
Homogeneous Dirichlet boundary conditions are used on the domain $[0,2\pi]\times[0, 2\pi]$. 
In order to test the designed spatial accuracy, a third order accurate IMEX backward differentiation formula (BDF) time discretization is used:  the nonlinear term is treated explicitly in time, and the convection diffusion terms are treated implicitly.
Errors at $T=0.2$ are listed in Table \ref{table-accuaracy}, in which we can observe the expected spatial order of accuracy. 
\begin{table}[h]
\centering
\caption{Accuracy test on uniform meshes for an Allen-Cahn equation. Third order IMEX BDF time discretization is used for time discretization. }
\resizebox{\textwidth}{!}{
\begin{tabular}{|c|c|c|c|c|c|c|c|c|}
\hline
\multirow{2}{*}{Finite Difference Grid} & 
\multicolumn{4}{c|}{second order scheme}  &
\multicolumn{4}{c|}{fourth order scheme} \\ 
\cline{2-9}
 & 
 $l^1$ error & order & $l^\infty$ error & order  &
$l^1$ error & order & $l^\infty$ error & order  \\
\hline
$9\times 9$  & 
        6.58E-2     & - &         2.38E-1     & - & 
        6.63E-2     & - &         2.66E-1     & - 
 \\
\hline
$19\times 19$ & 
        1.75E-2     &     1.91     &         8.80E-2     &     1.61     & 
        1.36E-2     &     2.28     &         5.23E-2     &     2.35     
 \\
\hline
$79\times 79$ &
    1.04E-3     &     2.02    &     4.75E-3     &     2.00     & 
        1.92E-5     &     4.85     &         1.21E-4     &     4.22     
 \\
\hline
$159\times 159$ &
    2.56E-4     &     2.02     &     1.19E-3     &     2.00     & 
    1.13E-6     &     4.09     &         7.15E-6     &     4.08     
 \\ \hline
\end{tabular}}
\label{table-accuaracy}
\end{table}

Next we test accuracy of two schemes solving the stream function vorticity equations of incompressible flow on the domain $[0,2\pi]\times [0, 2\pi ]$ with periodic boundary conditions. An exact solution 
$\omega=-2e^{-2\mu t} \sin x\sin y$ with $\mu=0.1$ is considered. 
In order to test the designed spatial accuracy, a third order accurate BDF time discretization is used.
Errors at $T=0.2$ are listed in Table \ref{table-accuaracy2}, in which we can observe the expected spatial order of accuracy. 
\begin{table}[h]
\centering
\caption{Accuracy test on uniform meshes for stream function vorticity equations with periodic boundary conditions. Third order BDF time discretization is used for time discretization. }
\resizebox{\textwidth}{!}{
\begin{tabular}{|c|c|c|c|c|c|c|c|c|}
\hline
\multirow{2}{*}{Finite Difference Grid} & 
\multicolumn{4}{c|}{second order scheme}  &
\multicolumn{4}{c|}{fourth order scheme} \\ 
\cline{2-9}
 & 
 $l^1$ error & order & $l^\infty$ error & order  &
$l^1$ error & order & $l^\infty$ error & order  \\
\hline
$40\times 40$  & 
        4.82E-5     & - &         1.14E-4     & - & 
        5.69E-5     & - &         2.30E-4     & - 
 \\
\hline
$80\times 80$ & 
        1.34E-5     &     1.84     &         3.23E-5     &     1.81     & 
        3.67E-6     &     3.96     &         1.51E-5     &     3.93     
 \\
\hline
$160\times 160$ &
        3.78E-6     &     1.83    &     9.21E-6     &     1.81     & 
        2.27E-7     &     4.01     &         9.47E-7     &     4.00     
 \\
\hline
$320\times 320$ &
    9.64E-7     &    1.97     &     2.36E-6     &     1.96     & 
    1.41E-8     &     4.00     &         5.91E-8     &     4.00     
 \\ \hline
\end{tabular}}
\label{table-accuaracy2}
\end{table}

 \begin{figure}[htbp]
 \subfigure[Second order scheme with first order IMEX on a $239\times 239$ grid]{\includegraphics[scale=0.32]{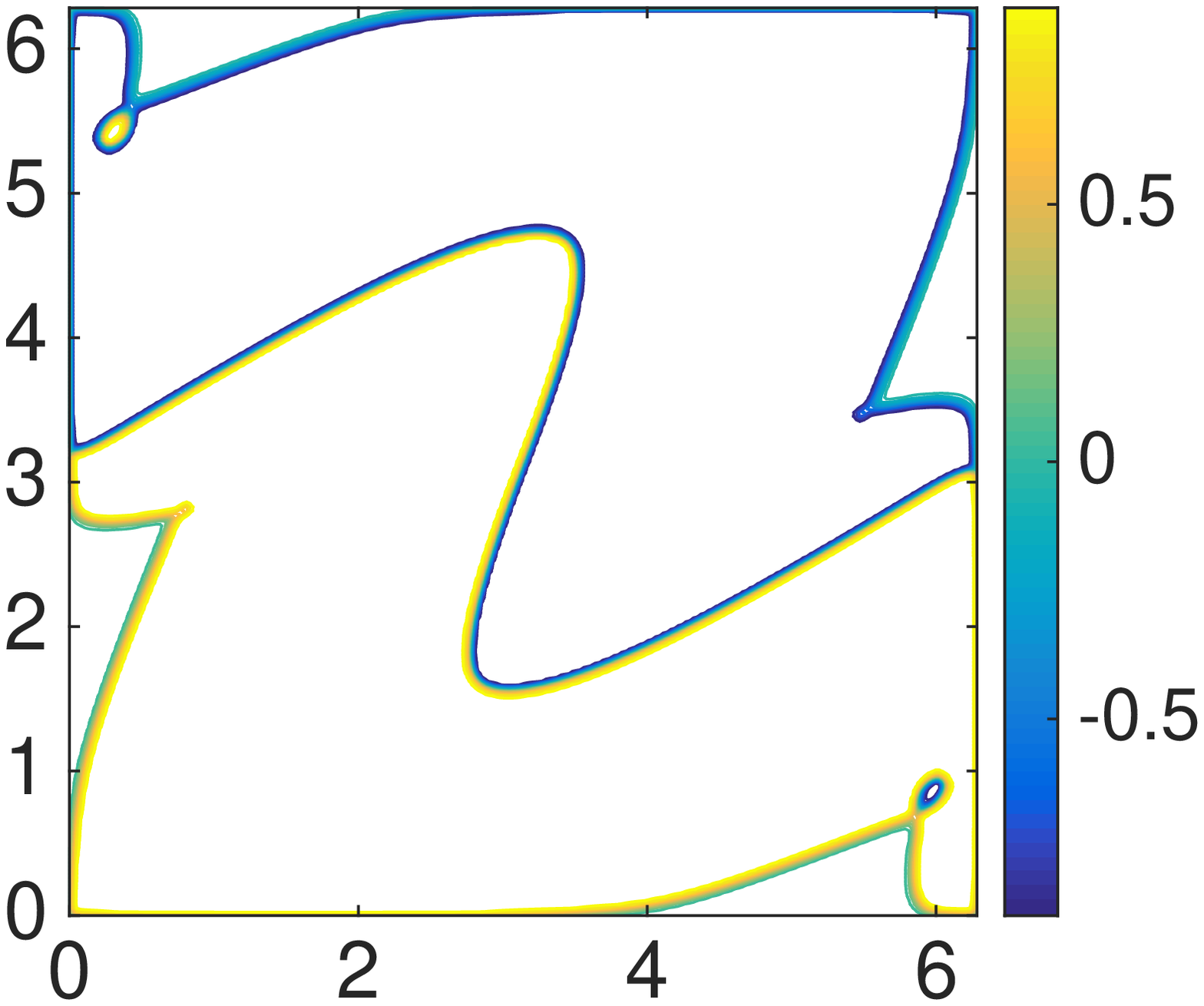} }
  \subfigure[Fourth order scheme with first order IMEX on a $239\times 239$ grid]{\includegraphics[scale=0.32]{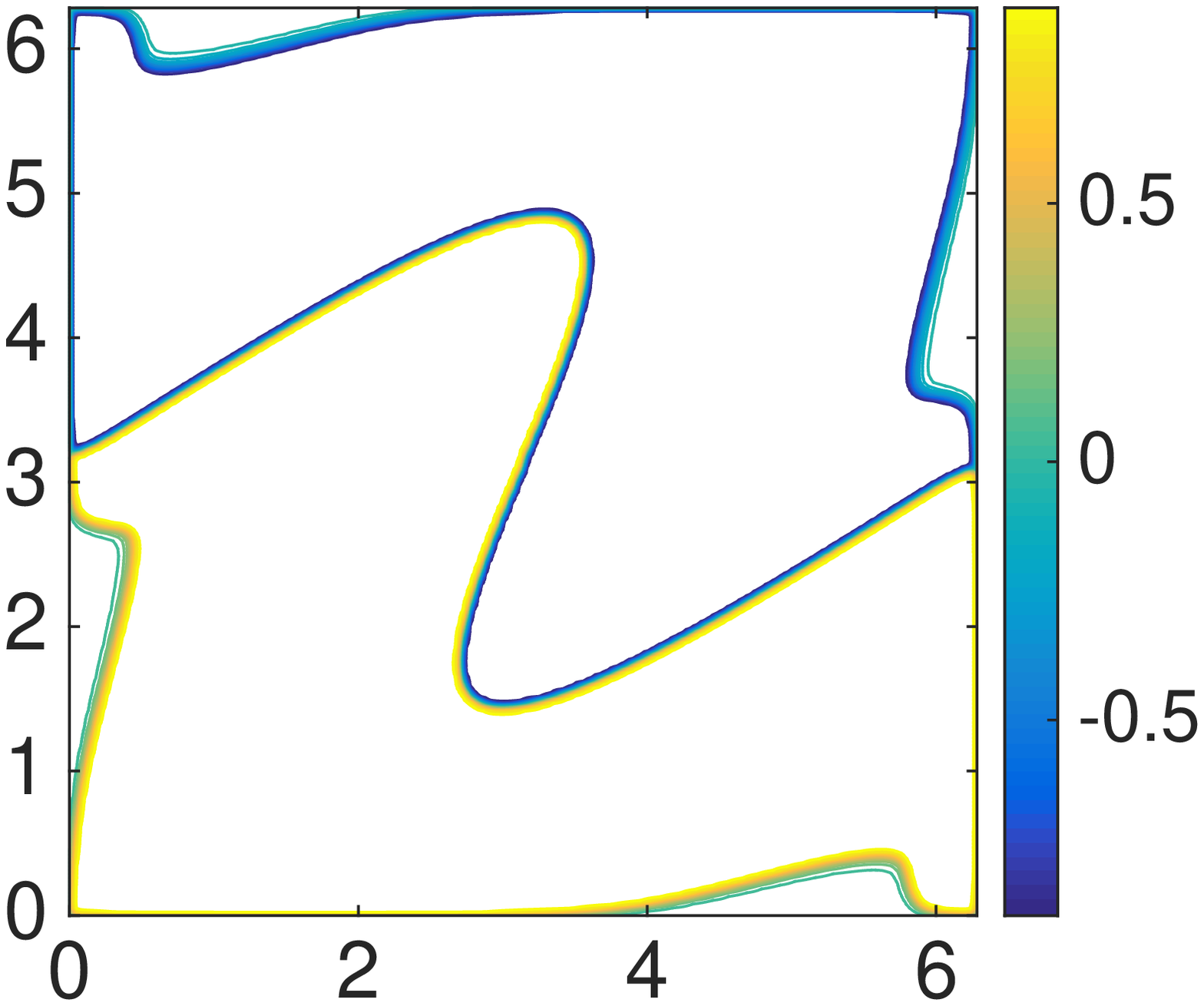} }\\
   \subfigure[Second order scheme with third order IMEX BDF on a $239\times 239$ grid]{\includegraphics[scale=0.32]{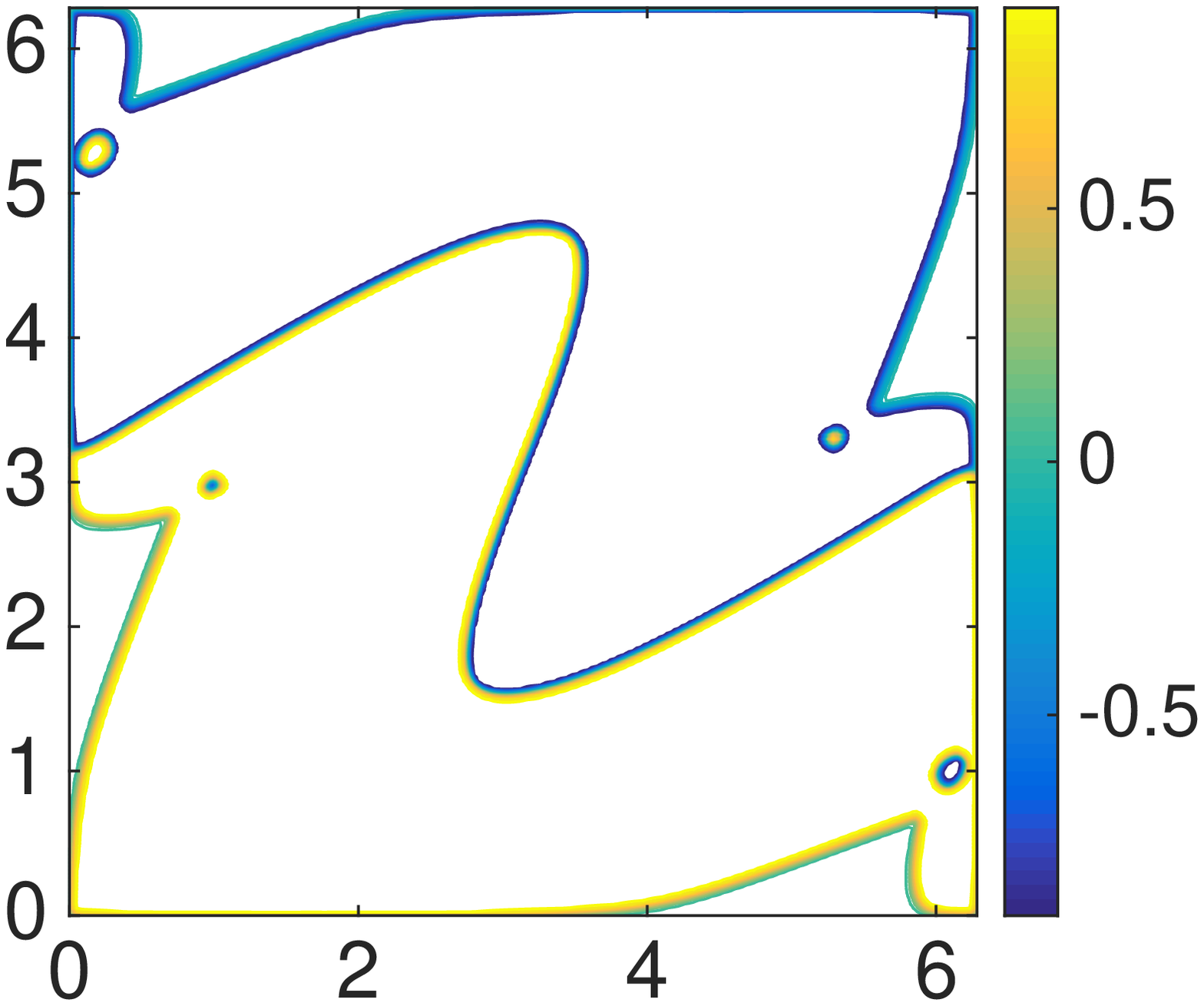} }
        \subfigure[Reference Solution]{\includegraphics[scale=0.32]{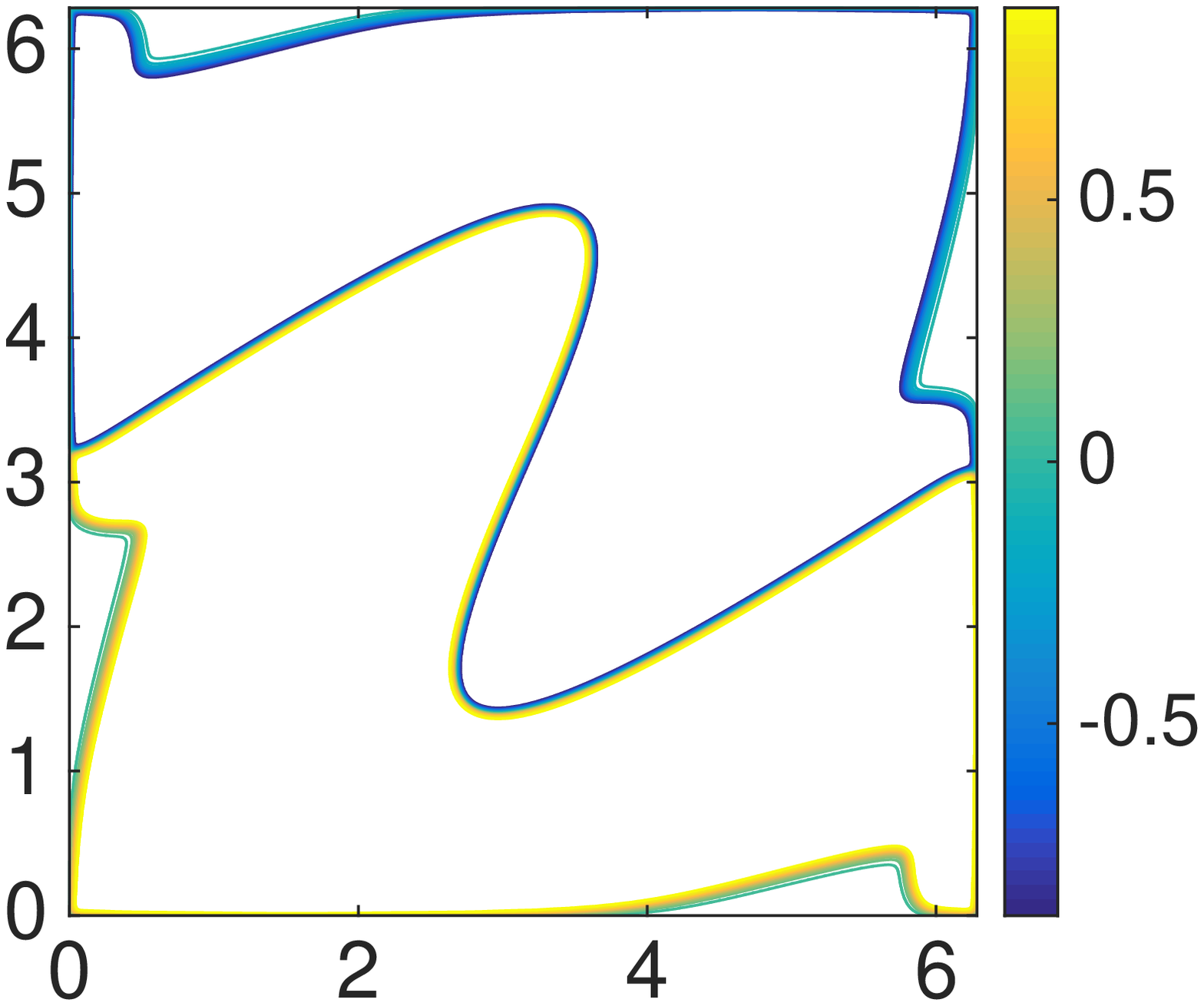} }
\caption{Allen-Cahn with log energy at $T=1.8$. The reference solution is generated by second order scheme with third order IMEX BDF  time discretization  on a $479\times 479$ grid.} 
\label{fig-Allen-Cahn-log}
 \end{figure}

 \begin{figure}[htbp]
 \subfigure[Second order scheme with first order IMEX on a $239\times 239$ grid]{\includegraphics[scale=0.32]{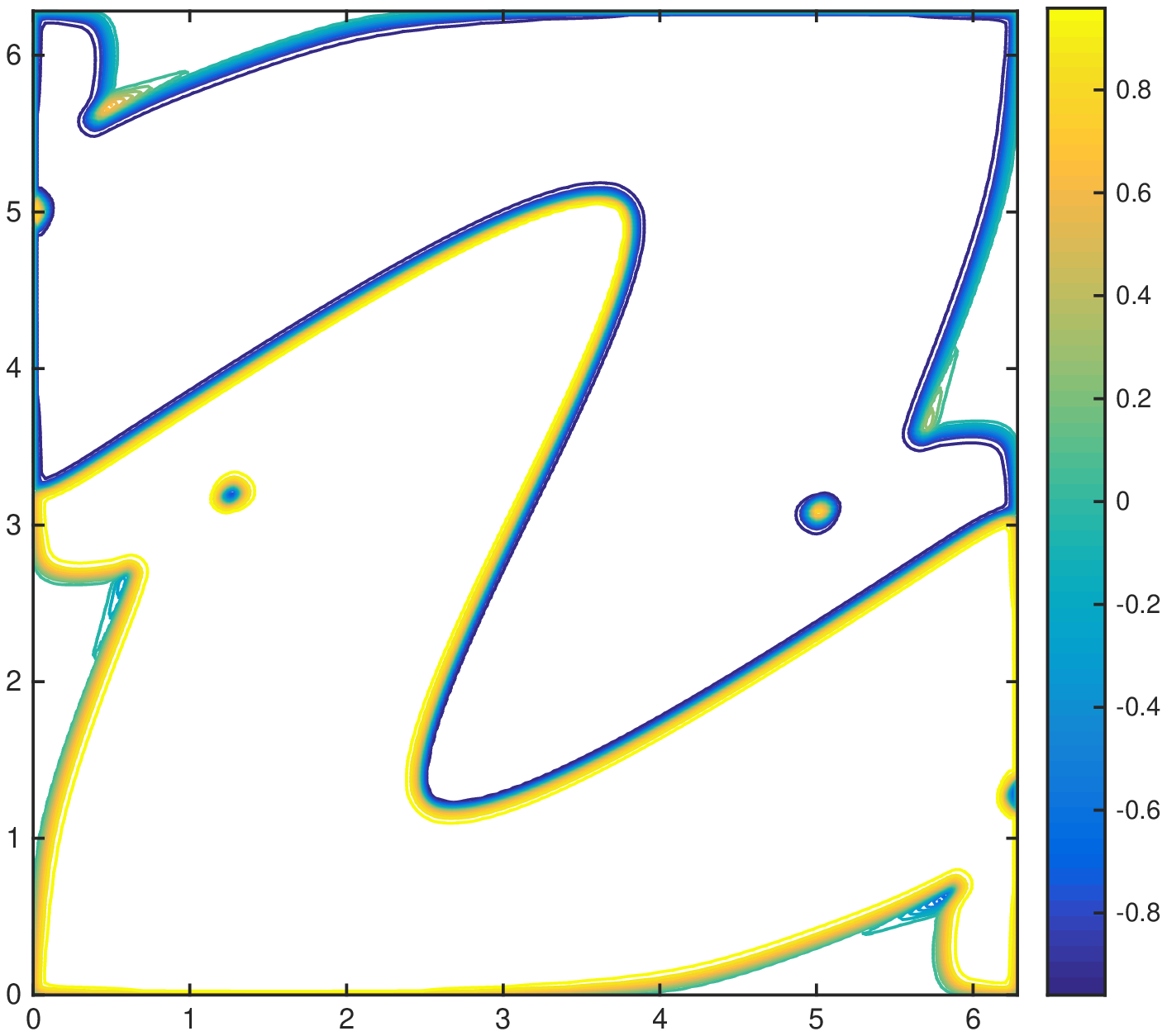} }
  \subfigure[Fourth order scheme with first order IMEX on a $239\times 239$ grid]{\includegraphics[scale=0.32]{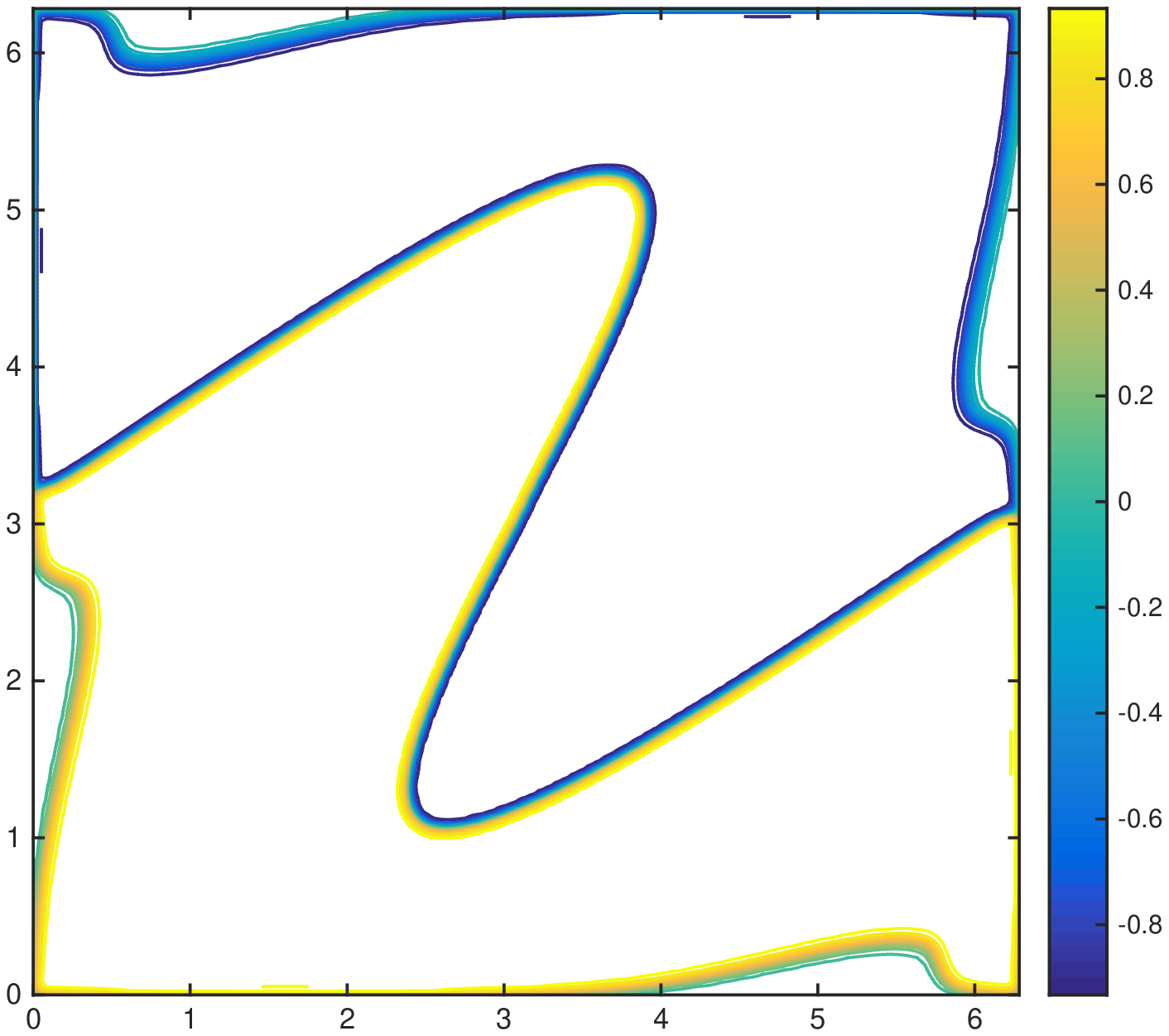} }\\
   \subfigure[Second order scheme with third order IMEX BDF on a $239\times 239$ grid]{\includegraphics[scale=0.32]{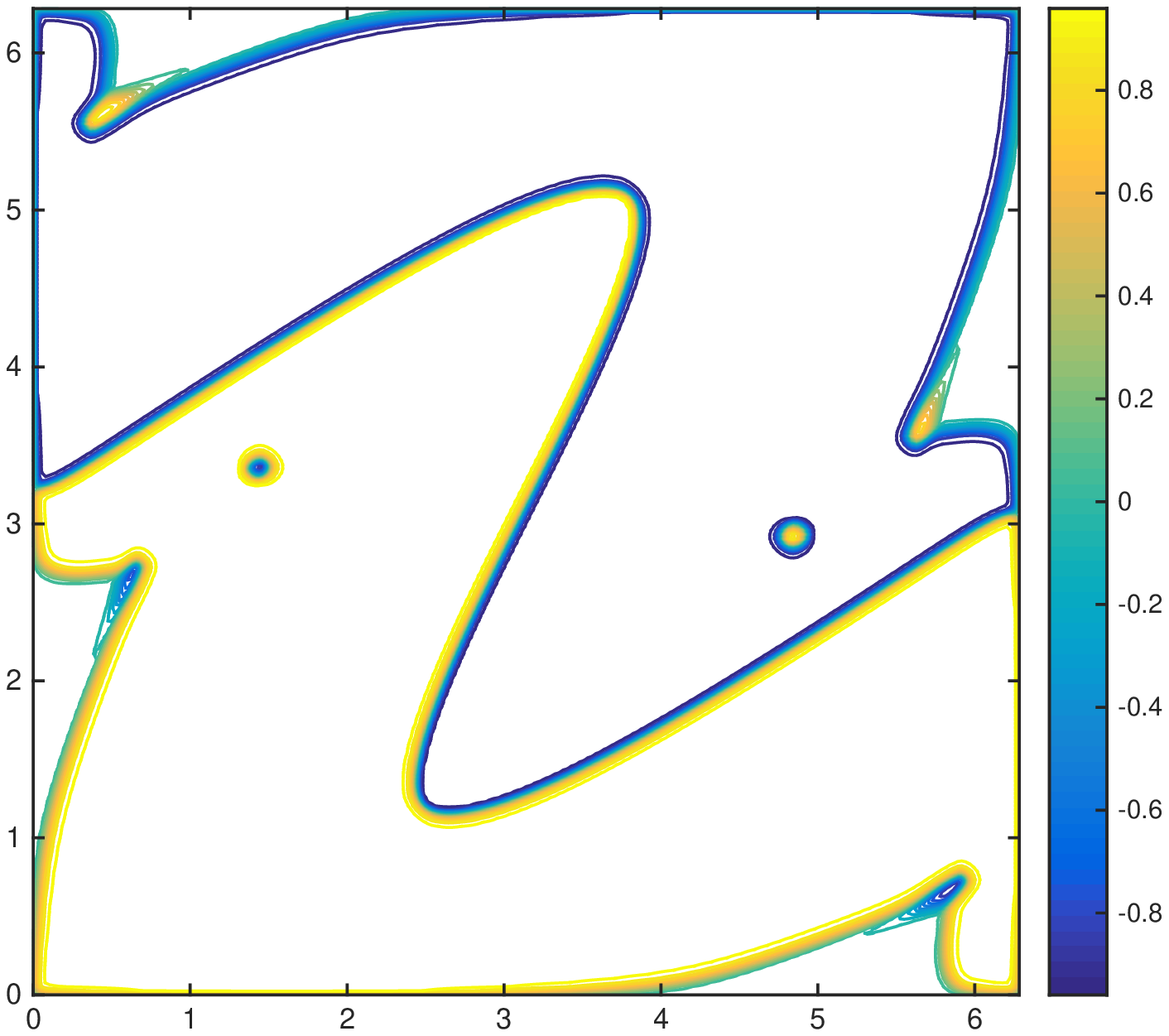} }
        \subfigure[Reference Solution]{\includegraphics[scale=0.32]{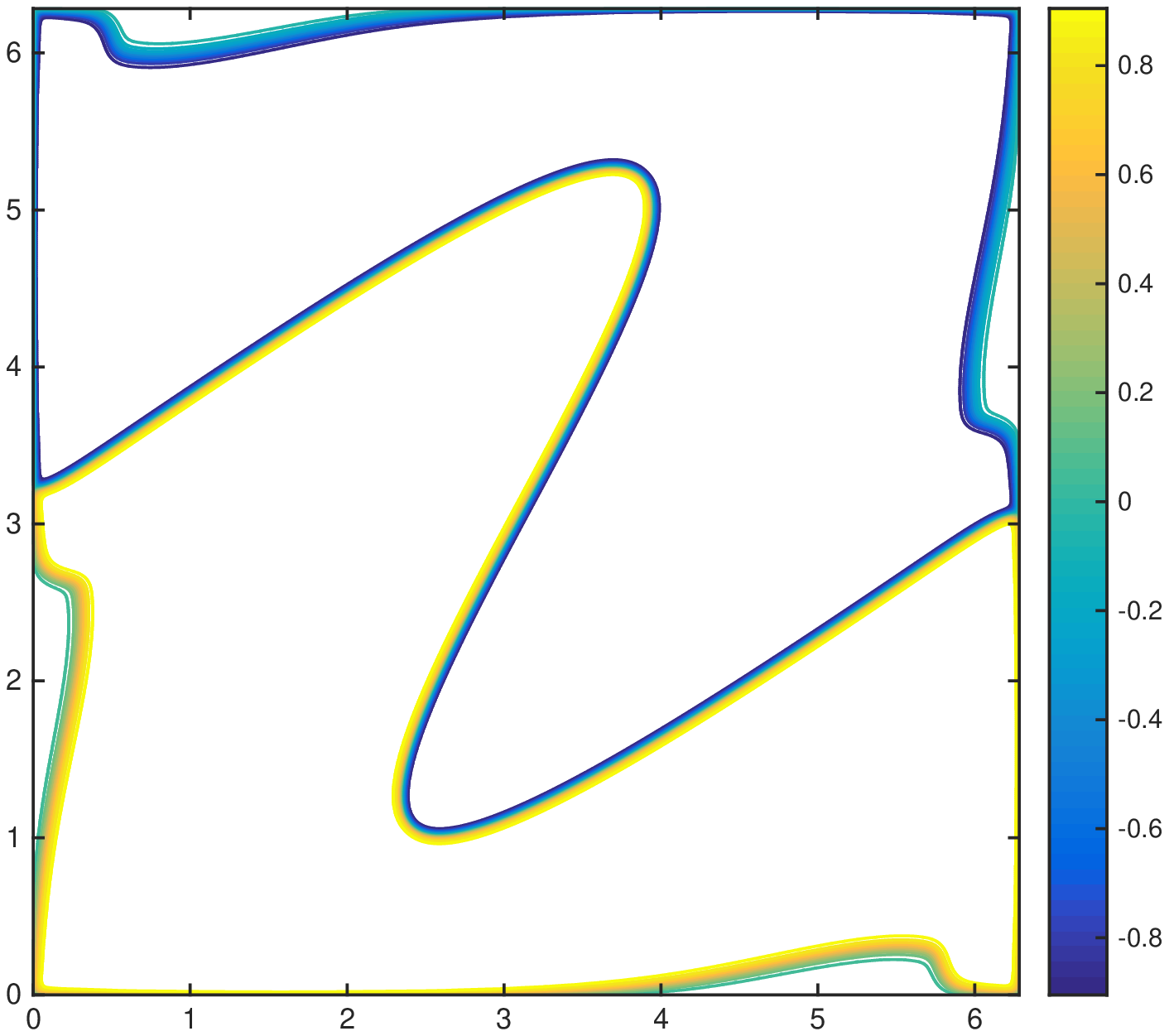} }
\caption{Allen-Cahn with polynomial energy at $T=2.2$. The reference solution is generated by second order scheme with third order IMEX BDF  time discretization on a $479\times 479$ grid.}
 
\label{fig-Allen-Cahn-poly}
 \end{figure}

\subsection{The generalized Allen-Cahn equation}

Next, we take a given velocity field  $u=v=\sin(y-x)$ in \eqref{generalized-Allen-Cahn} with a logarithmic energy function 
$$F(\phi)=\frac{\theta}{2}[(1+\phi)\ln(1+\phi)+(1-\phi) \ln(1-\phi)] -\frac{\theta_c}{2}\phi^2,$$
and parameters, $\theta=1$, $\theta_c=0.5$, $\mu=0.01$ and $\varepsilon=0.03$.
The initial condition is $\phi_0(x,y)=0.99\sin y \sin^2 x.$ 
The stability parameter $S=0$ is used, i.e., the time discretization is first order IMEX method. 
See Figure \ref{fig-Allen-Cahn-log} for performance of the schemes. 
We observe that the second order scheme with first order IMEX time discretization produces erroneous numerical artifacts on a relatively coarse $239\times 239$ grid, and higher order time discretization does not help reducing such an error. On the other hand, the fourth order scheme with first order IMEX method produces a satisfying solution on the $239\times 239$ grid. For both second order and fourth order schemes, the time step is taken as $\Delta t=\frac17\Delta x,$ and iterations needed for convergence 
in BiCGSTAB are almost the same for two schemes in each time step, thus the computational cost of both second order and fourth order schemes are almost the same on the same grid. Therefore, the fourth order scheme is obviously superior.

 Next we test a given velocity field  $u=v=\sin(y-x)$ for \eqref{generalized-Allen-Cahn} with a polynomial energy function $F(\phi)=\frac14(\phi^2-1)^2$, parameters $\mu=0.01$ and $\varepsilon=0.05$.
The initial condition is $\phi_0(x,y)=0.75\sin y \sin^2 x.$ 
The stability parameter $S=0$ is used, i.e., the time discretization is first order IMEX method. 
See Figure \ref{fig-Allen-Cahn-poly} for performance of the schemes. 
We observe that the second order scheme with first order IMEX time discretization produces erroneous numerical artifacts on a relatively coarse $239\times 239$ grid, and higher order time discretization does not help reducing such an error. On the other hand, the fourth order scheme with first order IMEX method produces a satisfying solution on the $239\times 239$ grid. For both second order and fourth order schemes, the time step is taken as $\Delta t=\frac16\Delta x,$ and iterations needed for convergence 
in BiCGSTAB are almost the same for two schemes in each time step, thus the computational cost of both second order and fourth order schemes are almost the same on the same grid. Therefore, the fourth order scheme is obviously superior.

\subsection{Incompressible flow: double shear layer}
We test the same schemes for solving the incompressible Navier-Stokes system with $\mu=0.001$ consisting of a scalar convection-diffusion for vorticity and a Poisson equation for stream function, as described in Section \ref{sec-ns}. 
We consider the following initial condition with the periodic boundary conditions on $[0, 2\pi]\times[0, 2\pi]$:
\[ \omega(x,y,0)=\begin{cases}
\delta \cos x-\frac{1}{\rho}sech^2\frac{y-\frac{\pi}{2}}{\rho},& y\leq \pi\\
\delta \cos x+\frac{1}{\rho}sech^2\frac{\frac{3\pi}{2}-y}{\rho},& y> \pi\\
\end{cases}\]
with $\rho=\frac{\pi}{15}$ and $\delta=0.05.$
This is a classical test for 2D incompressible Navier Stokes in vorticity form.  
See Figure \ref{test3} and 
Figure \ref{test4} for the performance of the schemes. 
For both schemes, the time step is taken as $\Delta t=\frac{1}{6\|\mathbf u\|_\infty}\Delta x$ and iterations needed for convergence 
in BiCGSTAB are almost the same, thus the computational cost  are almost the same. 
Similar to observations for the Allen-Cahn equation, 
we can see that  the second order scheme with first order IMEX time discretization produces erroneous numerical oscillations on a relatively coarse $120\times 120$ grid, and higher order time discretization does not help reducing such an error. The fourth order scheme produces much better solutions on the $120\times 120$ grid.

 \begin{figure}[htbp]
 \subfigure[Second order difference scheme with first order IMEX on a $120\times 120$ grid]{\includegraphics[scale=0.32]{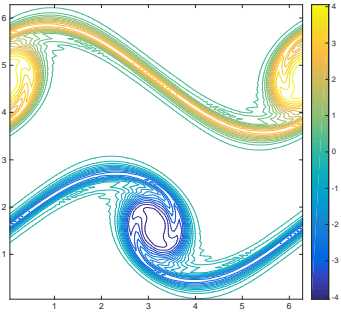} }
  \subfigure[Fourth order difference scheme with first order IMEX  on a $120\times 120$ grid]{\includegraphics[scale=0.32]{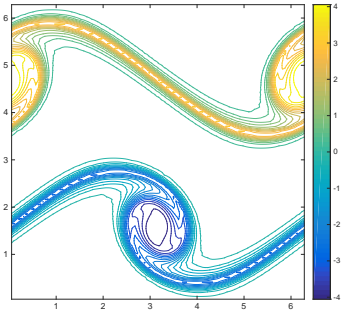} }\\
      \subfigure[Second order difference scheme with third order IMEX BDF on a $120\times 120$ grid]{\includegraphics[scale=0.32]{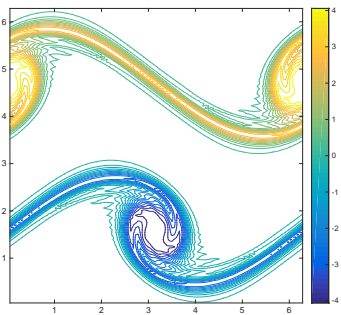} }
    \subfigure[Reference Solution]{\includegraphics[scale=0.32]{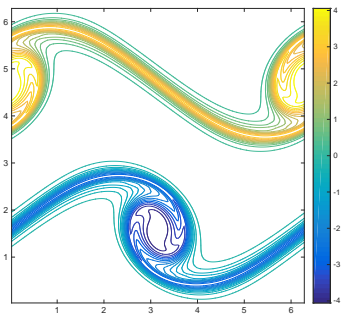} }
\caption{The 2D incompressible Navier-Stokes in vorticity form at $T=6$ with $\mu=0.001$. The reference solution is generated by second order difference scheme with third order IMEX BDF  time discretization on a $240\times 240$ grid.}
\label{test3}
 \end{figure}
 
  \begin{figure}[htbp]
 \subfigure[Second order difference scheme with first order IMEX on a $120\times 120$ grid]{\includegraphics[scale=0.32]{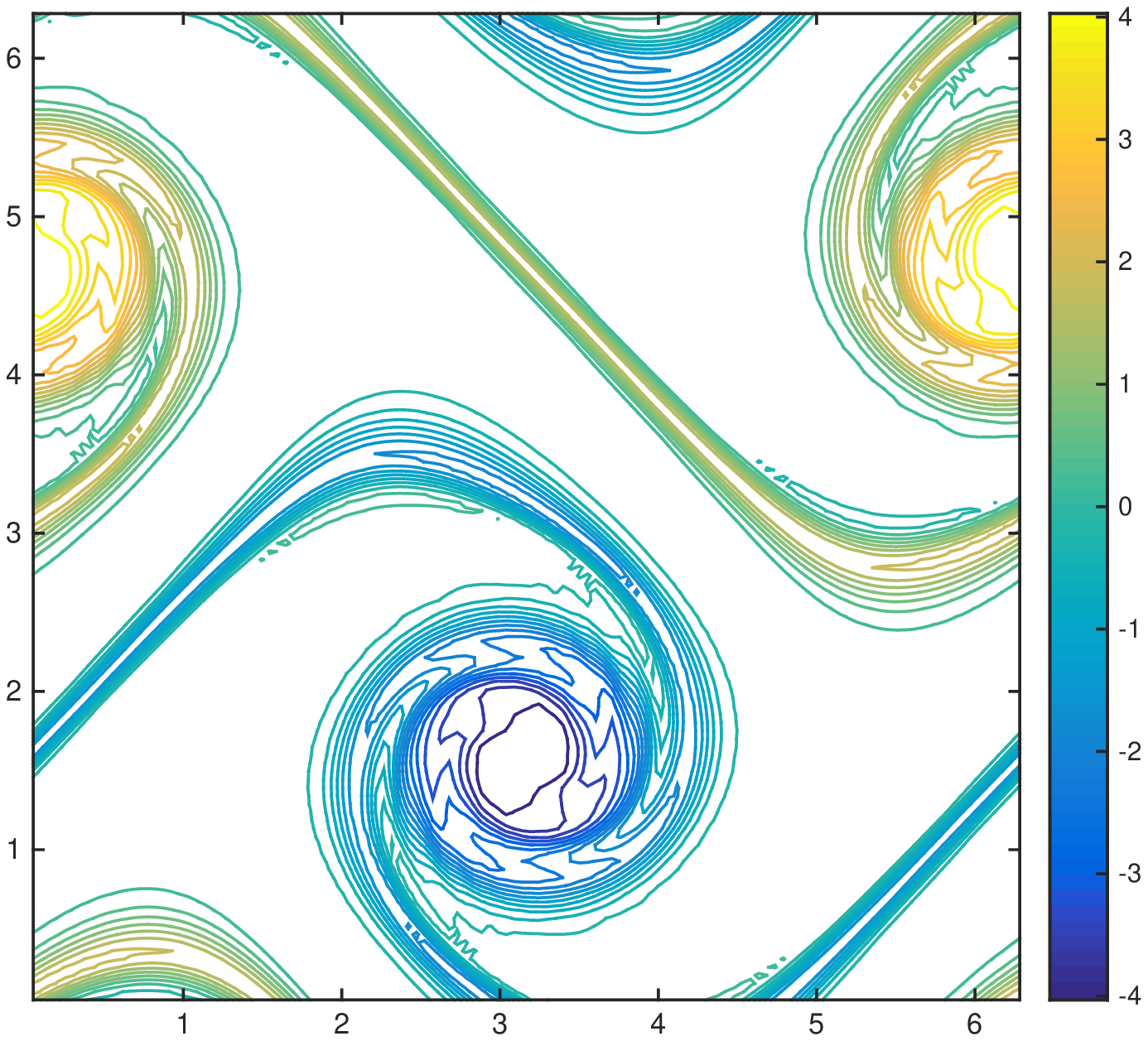} }
  \subfigure[Fourth order difference scheme with first order IMEX  on a $120\times 120$ grid]{\includegraphics[scale=0.32]{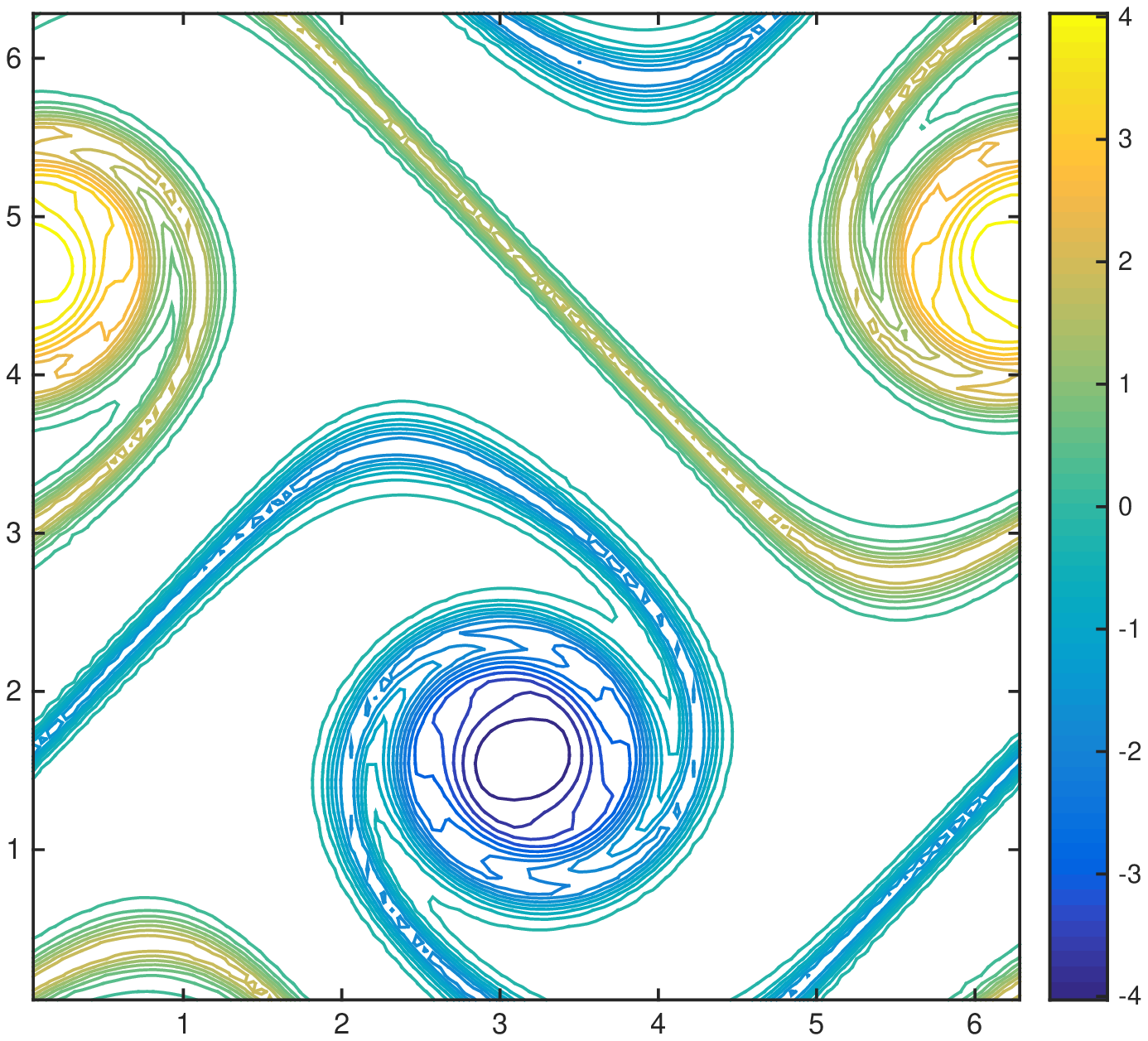} }\\
      \subfigure[Second order difference scheme with third order IMEX BDF on a $120\times 120$ grid]{\includegraphics[scale=0.32]{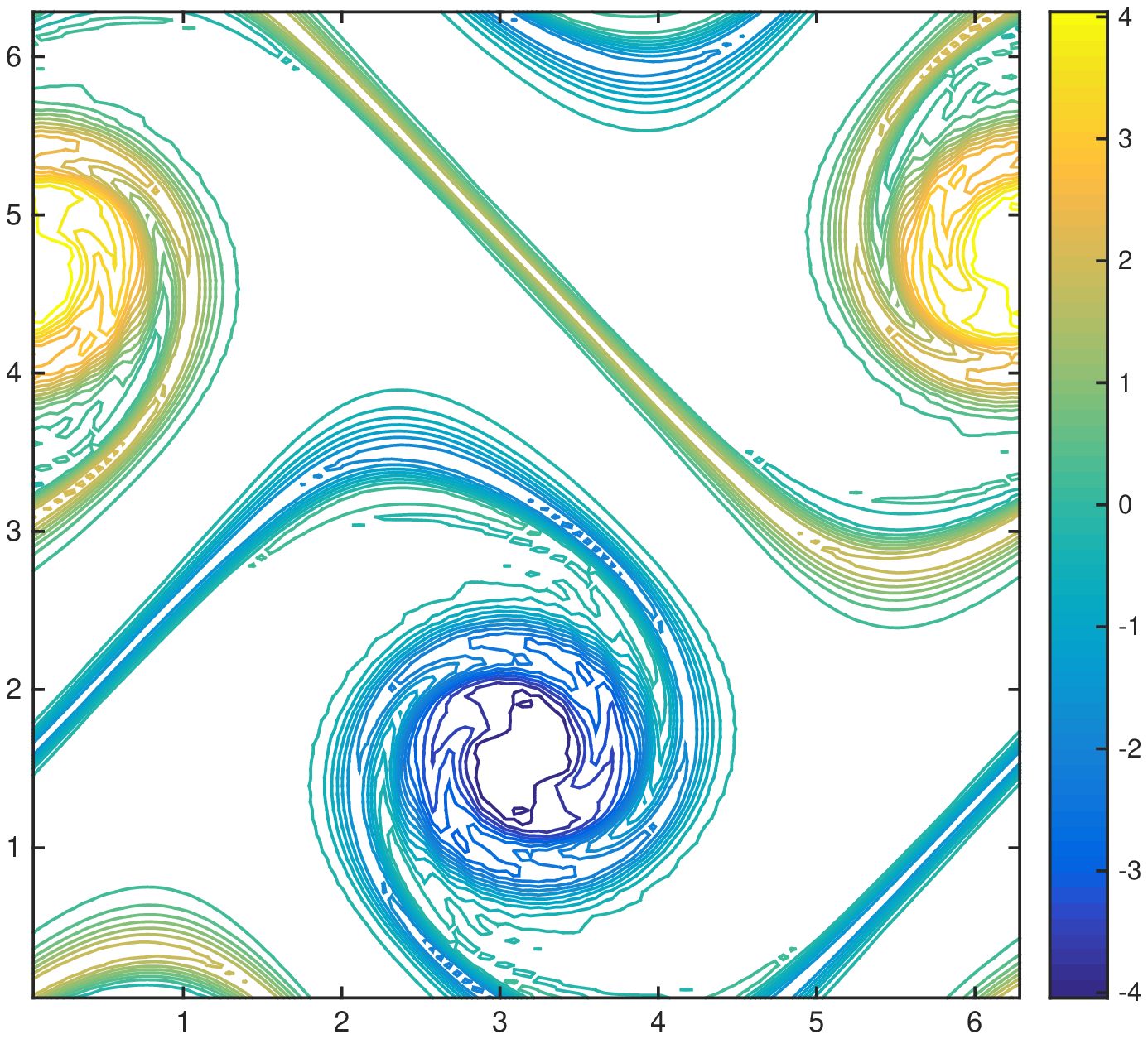} }
    \subfigure[Reference Solution]{\includegraphics[scale=0.32]{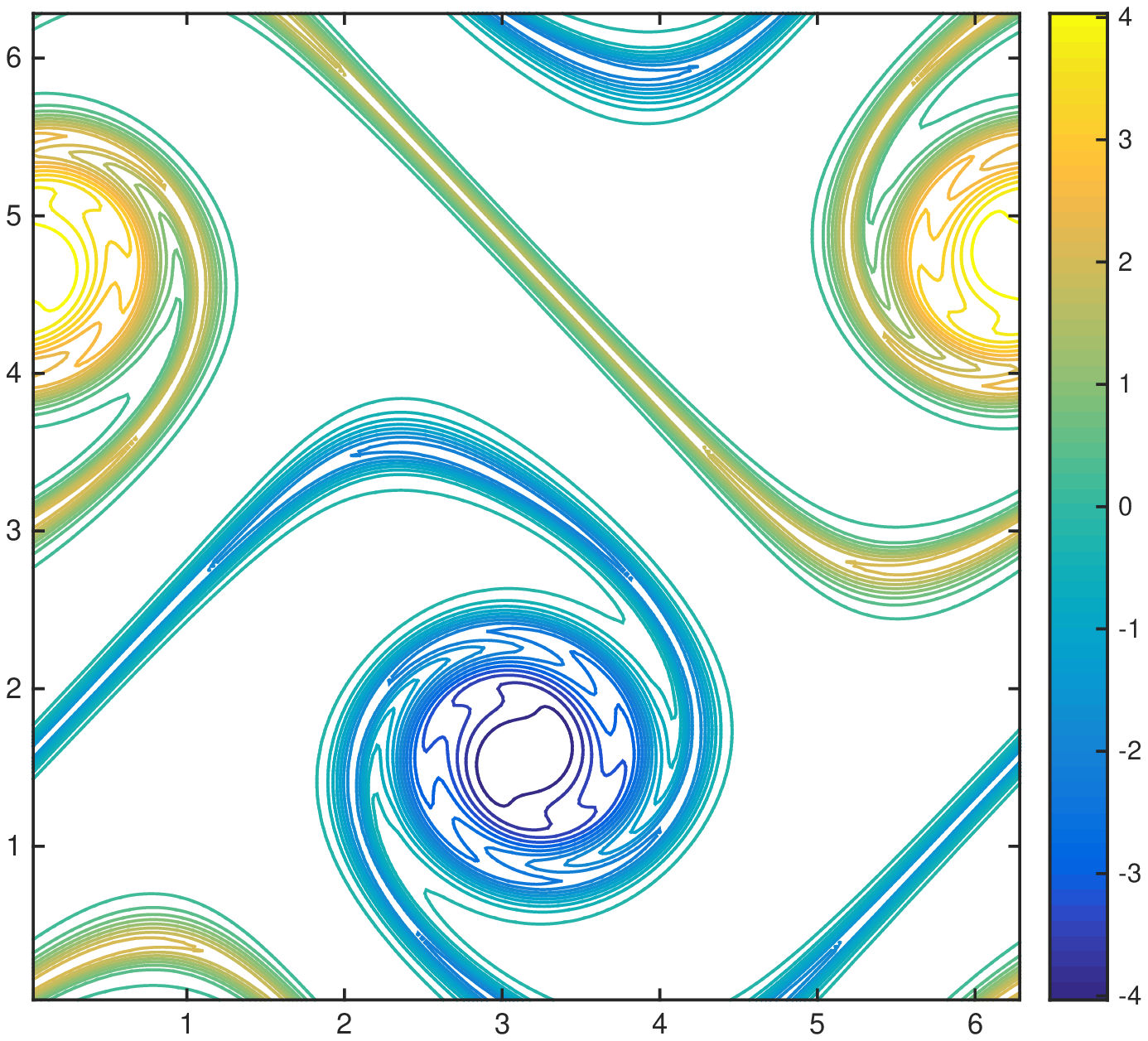} }
\caption{The 2D incompressible Navier-Stokes in vorticity form at $T=8$ with $\mu=0.001$. The reference solution is generated by second order difference scheme with third order IMEX BDF  time discretization  on a $240\times 240$ grid.}
\label{test4}
 \end{figure}

 \section{Concluding remarks}
 \label{sec-remark}
 
 In this paper we have proven the monotonicity of the finite difference implementation of the $Q^2$ spectral element method for a linear convection-diffusion operator with a given incompressible velocity field. Thanks to the monotonicity, we obtained a fourth order accurate finite difference spatial discretization satisfying the discrete maximum principle, and used it  to construct bound-preserving schemes for the generalized Allen-Cahn equation. 
 To the best of our knowledge, this is the first time that a high order spatial discretization with an IMEX discretization in time is proven to satisfy a discrete maximum principle for  a linear convection-diffusion operator. 
We presented several  numerical tests which showed  superiority of higher order spatial accuracy compared to the most popular bound-preserving second order scheme. 
Even though we only discussed the monotonicity for the fourth order finite difference scheme solving the two-dimensional problem, it is straightforward to extend the discussion of monotonicity to a three-dimensional linear convection-diffusion problem with an incompressible velocity field. 
\newpage

\bibliographystyle{siamplain}
\bibliography{references,ref}

\begin{thebibliography}{10}

\bibitem{All.C79}
{\sc S.~M. Allen and J.~W. Cahn}, {\em A microscopic theory for antiphase
  boundary motion and its application to antiphase domain coarsening}, Acta
  metallurgica, 27 (1979), pp.~1085--1095.

\bibitem{chen2002phase}
{\sc L.-Q. Chen}, {\em Phase-field models for microstructure evolution}, Annual
  review of materials research, 32 (2002), pp.~113--140.

\bibitem{chen2016third}
{\sc Z.~Chen, H.~Huang, and J.~Yan}, {\em {Third order
  maximum-principle-satisfying direct discontinuous Galerkin methods for time
  dependent convection diffusion equations on unstructured triangular meshes}},
  Journal of Computational Physics, 308 (2016), pp.~198--217.

\bibitem{cross2020monotonicity}
{\sc L.~J. Cross and X.~Zhang}, {\em {On the monotonicity of high order
  discrete Laplacian}}, arXiv preprint arXiv:2010.07282,  (2020).

\bibitem{du2019}
{\sc Q.~Du, L.~Ju, X.~Li, and Z.~Qiao}, {\em Maximum principle preserving
  exponential time differencing schemes for the nonlocal allen--cahn equation},
  SIAM Journal on Numerical Analysis, 57 (2019), pp.~875--898.

\bibitem{gottlieb2009high}
{\sc S.~Gottlieb, D.~I. Ketcheson, and C.-W. Shu}, {\em High order strong
  stability preserving time discretizations}, Journal of Scientific Computing,
  38 (2009), pp.~251--289.

\bibitem{GLY19}
{\sc L.~Guo, X.~Li, and Y.~Yang}, {\em {Energy dissipative local discontinuous
  Galerkin methods for Keller-Segel chemotaxis model}}, J. Sci. Comput., 78
  (2019), pp.~1387--1404.

\bibitem{li2021accuracy}
{\sc H.~Li, D.~Appel\"o, and X.~Zhang}, {\em {Accuracy of spectral element
  method for wave, parabolic and Schr\"{o}dinger equations}}, arXiv preprint
  arXiv:2103.00400,  (2021).

\bibitem{li2018high}
{\sc H.~Li, S.~Xie, and X.~Zhang}, {\em A high order accurate bound-preserving
  compact finite difference scheme for scalar convection diffusion equations},
  SIAM Journal on Numerical Analysis, 56 (2018), pp.~3308--3345.

\bibitem{li2019monotonicity}
{\sc H.~Li and X.~Zhang}, {\em {On the monotonicity and discrete maximum
  principle of the finite difference implementation of $C^0$-$ Q^2$ finite
  element method}}, Numerische Mathematik,  (2020), pp.~1--36.

\bibitem{li2020superconvergence}
{\sc H.~Li and X.~Zhang}, {\em Superconvergence of high order finite difference
  schemes based on variational formulation for elliptic equations}, Journal of
  Scientific Computing, 82 (2020), p.~36.

\bibitem{li2019fourth}
{\sc H.~Li and X.~Zhang}, {\em Superconvergence of high order finite difference
  schemes based on variational formulation for elliptic equations}, Journal of
  Scientific Computing, 82 (2020), p.~36.

\bibitem{liu2003phase}
{\sc C.~Liu and J.~Shen}, {\em {A phase field model for the mixture of two
  incompressible fluids and its approximation by a Fourier-spectral method}},
  Physica D: Nonlinear Phenomena, 179 (2003), pp.~211--228.

\bibitem{lorenz1977inversmonotonie}
{\sc J.~Lorenz}, {\em Zur inversmonotonie diskreter probleme}, Numerische
  Mathematik, 27 (1977), pp.~227--238.

\bibitem{maday1990optimal}
{\sc Y.~Maday and E.~M. R{\o}nquist}, {\em Optimal error analysis of spectral
  methods with emphasis on non-constant coefficients and deformed geometries},
  Computer Methods in Applied Mechanics and Engineering, 80 (1990),
  pp.~91--115.

\bibitem{plemmons1977m}
{\sc R.~J. Plemmons}, {\em {M-matrix characterizations. I----nonsingular
  M-matrices}}, Linear Algebra and its Applications, 18 (1977), pp.~175--188.

\bibitem{qiu2021third}
{\sc C.~Qiu, Q.~Liu, and J.~Yan}, {\em {Third order positivity-preserving
  direct discontinuous Galerkin method with interface correction for chemotaxis
  Keller-Segel equations}}, Journal of Computational Physics,  (2021),
  p.~110191.

\bibitem{shen2016maximum}
{\sc J.~Shen, T.~Tang, and J.~Yang}, {\em {On the maximum principle preserving
  schemes for the generalized Allen--Cahn equation}}, Communications in
  Mathematical Sciences, 14 (2016), pp.~1517--1534.

\bibitem{srinivasan2018positivity}
{\sc S.~Srinivasan, J.~Poggie, and X.~Zhang}, {\em {A positivity-preserving
  high order discontinuous Galerkin scheme for convection--diffusion
  equations}}, Journal of Computational Physics, 366 (2018), pp.~120--143.

\bibitem{sun2018discontinuous}
{\sc Z.~Sun, J.~A. Carrillo, and C.-W. Shu}, {\em {A discontinuous Galerkin
  method for nonlinear parabolic equations and gradient flow problems with
  interaction potentials}}, Journal of Computational Physics, 352 (2018),
  pp.~76--104.

\bibitem{yang2006numerical}
{\sc X.~Yang, J.~J. Feng, C.~Liu, and J.~Shen}, {\em Numerical simulations of
  jet pinching-off and drop formation using an energetic variational
  phase-field method}, Journal of Computational Physics, 218 (2006),
  pp.~417--428.

\bibitem{zhang2012maximum}
{\sc X.~Zhang, Y.~Liu, and C.-W. Shu}, {\em {Maximum-principle-satisfying high
  order finite volume weighted essentially nonoscillatory schemes for
  convection-diffusion equations}}, SIAM Journal on Scientific Computing, 34
  (2012), pp.~A627--A658.

\end{thebibliography}

\end{document}